\documentclass[11pt]{article}
\usepackage[latin1]{inputenc}
\usepackage{amscd}
\usepackage{amsfonts}
\usepackage{amsmath}
\usepackage{amssymb}
\usepackage{amsthm}
\usepackage{bbm}
\usepackage{CJK}
\usepackage{fancyhdr}
\usepackage{graphicx}
\usepackage{hyperref}
\usepackage{indentfirst}
\usepackage{latexsym}
\usepackage{mathrsfs}
\usepackage{xypic}

\usepackage[top=1in,bottom=1in,left=1.25in,right=1.25in]{geometry}
\def\<{\langle}
\def\>{\rangle}
\def\a{\alpha}
\def\b{\beta}

\def\d{\delta}

\def\g{\gamma}

\def\v{\epsilon }

\newtheorem{df}{Definition}[section]
\newtheorem{thm}{Theorem}[section]
\newtheorem{cor}{Corollary}[section]
\newtheorem{rem}{Remark}[section]

\newtheorem{prop}{Proposition}[section]
\newtheorem{exa}{Example}[section]
\newtheorem{lem}{Lemma}[section]

\let\dis= \displaystyle

\date{}
\begin{document}
\renewcommand{\baselinestretch}{1.2}
\renewcommand{\arraystretch}{1.0}
\title{\bf  BiHom-alternative, BiHom-Malcev   and BiHom-Jordan algebras}
\author{{\bf Taoufik Chtioui$^{1}$, Sami Mabrouk$^{2}$, Abdenacer Makhlouf$^{3}$ }\\
{\small 1.  University of Sfax, Faculty of Sciences Sfax,  BP
1171, 3038 Sfax, Tunisia} \\
{\small 2.  University of Gafsa, Faculty of Sciences Gafsa,  2112 Gafsa, Tunisia}\\
{\small 3.~ IRIMAS - département de Mathématiques,
6, rue des frères Lumière, F-68093 Mulhouse, France}}
 \maketitle
 \maketitle

\begin{abstract}
The purpose of this paper is to introduce and study  BiHom-alternative algebras and BiHom-Malcev  algebras. It is shown that BiHom-alternative algebras are BiHom-Malcev  admissible and BiHom-Jordan admissible. Moreover, BiHom-type generalizations of some well known
identities in alternative algebras, including the Moufang identities, are obtained.
\end{abstract}
\begin{small}

 \end{small}
 \normalsize\vskip1cm

\section*{Introduction}
\renewcommand{\theequation}{\thesection.\arabic{equation}}
 An alternative algebra is an algebra whose associator is an alternating function. In particular, all associative algebras are alternative. The class of 8-dimensional Cayley
algebras (or Cayley-Dickson algebras, the prototype having been discovered
in 1845 by Cayley and later generalized by Dickson) is an important class of alternative algebras which are not associative.
Another important class of nonassociative algebras is the Malcev  algebras which were introduced in 1955 by A.I. Malcev  \cite{maltsev}, 
generalizing Lie algebras.  They  play an important role in Physics and the geometry of smooth loops.   Just as the tangent algebra of a Lie group is a Lie algebra, the tangent algebra of a locally analytic Moufang loop is a Malcev  algebra \cite{kerdman,kuzmin,maltsev,nagy,sabinin}, see also   \cite{gt,myung,okubo} for discussions about connections with physics. Jordan algebras were also introduced in Physics  context by Pascual Jordan in 1933 to provide a new formalism for Quantum Mechanics. The theory was developed from the algebraic viewpoint  by Jacobson in \cite{jacobson}. 

Alternative algebras are  closely related to Jordan algebras \cite{gt,jvw,okubo,sv}  and Malcev  algebras \cite{maltsev} in the same way that associative algebras are related to Jordan and Lie algebras.
Indeed, as Malcev  observed in \cite{maltsev}, every alternative algebra A is Malcev-admissible, i.e., the
commutator algebra $A^-$ is a Malcev  algebra.
On the other hand, starting with an alternative algebra $A$, it is known that the Jordan product (named also the anti-commutator)
gives a Jordan algebra. In other words, alternative algebras are Jordan-admissible.

Algebras where the identities defining the structure are twisted by a homomorphism are called
Hom-algebras. They have been intensively investigated in the literature recently. The theory of Hom-algebra
started from Hom-Lie algebras introduced and discussed in \cite{hls, LarssonSilvestrov}, motivated
by quasi-deformations of Lie algebras of vector fields, in particular q-deformations of Witt and
Virasoro algebras. Hom-associative algebras were introduced in \cite{ms} while Hom-alternative and
Hom-Jordan algebras were introduced in \cite{mak, Yau3} as twisted generalizations of alternative and Jordan
algebras respectively. Motivated by a categorical study of Hom-algebra and new type of categories,  
the authors introduced In
\cite{GrazianiMakhloufMeniniPanaite}  a generalized algebraic structure dealing with two commuting
multiplicative linear maps, called BiHom-algebras including 
BiHom-associative algebras and  BiHom-Lie algebras. When the two linear maps are the 
same, then BiHom-algebras will be turn to Hom-algebras in some cases. 

Th aim of this paper to consider  the BiHom-type of alternative, Malcev  and Jordan algebras.
We discuss some of their
properties and provide construction procedures using ordinary alternative algebras, Malcev  algebras or Jordan algebras.
Also, we show that every regular BiHom-alternative algebra is BiHom-Malcev  admissible and BiHom-Jordan admissible.

In the first section, we introduce the notion of left (resp. right) BiHom-alternative algebra and study the properties.  We also, provide  a construction Theorem of BiHom-alternative algebras.  We show that an ordinary alternative algebra along with two commuting morphisms lead to a BiHom-alternative algebra. In Section, we deal with BiHom-Malcev algebra and give similar Twisting construction. Moreover, we prove that regular BiHom-alternative algebras are BiHom-Malcev-admissible and study BiHom-type Bruck-Kleinfeld function.  In Section 4, we discuss further properties and BiHom-Moufang identities. In the last Section, we show that BiHom-alternative algebras are BiHom-Jordan-admissible.

\section{BiHom-Alternative algebras}


Throughout this paper $\mathbb{K}$ is an algebraically closed field of characteristic $0$ and $A$ is a $\mathbb{K}$-vector space. In the sequel, a BiHom-algebra refers to a quadruple $(A,\mu,\alpha,\beta)$, where $\mu: A\otimes A\rightarrow A$,   $\alpha: A\rightarrow A$ and $\beta: A\rightarrow A$ are linear maps. The composition of maps is denoted by concatenation for simplicity.

In this section, we introduce BiHom-alternative algebras,  study their general properties and  provide some construction results.

\subsection{Definitions and Properties }

\begin{df}
Let $(A,\mu,\alpha,\beta)$ be a BiHom-algebra.
\begin{enumerate}
  \item The  BiHom-associator of $A$  is the trilinear map $as_{\alpha,\beta}:A^{\otimes 3} \longrightarrow A$ defined by
\begin{equation}\label{BiHomAss}
    as_{\alpha,\beta}=\mu \circ (\mu\otimes \beta- \alpha \otimes \mu).
\end{equation}
In terms of elements, the map $as_{\alpha,\beta}$ is given by
$$ as_{\alpha,\beta}(x,y,z)=\mu(\mu(x,y),\beta(z))-\mu(\alpha(x),\mu(y,z)),\ \forall x,y,z \in A$$
  \item  The BiHom-Jacobiator of $A$  is the trilinear map  $J_{\alpha,\beta}: A^{\times3} \longrightarrow A$ defined as
\begin{equation}\label{BiHomJacobi}
  J_{\alpha,\beta}(x,y,z) = \circlearrowleft_{x,y,z}\mu(\beta^2(x),\mu(\beta(y),\alpha(z)))
\end{equation}
\end{enumerate}
\end{df}
Note that when $(A,\mu)$ is a algebra $($with $\alpha=\beta=id )$, its BiHom-associator and BiHom-Jacobiator coincide with its usual associator and Jacobiator, respectively.  

\begin{df}
A BiHom-associative algebra is a quadruple $(A,\mu,\alpha,\beta)$
where $\alpha,\beta: A \rightarrow A$, are linear maps and $\mu : A\times A\rightarrow A$ is a bilinear map such that  $\alpha  \beta =\beta \alpha$, $\alpha\mu=\mu\alpha^{\otimes2}$, $\beta\mu=\mu\beta^{\otimes 2}$ and
satisfying the following condition:
\begin{equation}\label{BiHom ass identity}
    as_{\alpha,\beta}(x,y,z)=0,\ \text{for all}\ x,y,z \in A\ \textrm{(BiHom-associativity condition).}
\end{equation}
The maps $\alpha$, $\beta$ are called the structure maps of $A$.
\end{df}

Clearly, a Hom-associative algebra $(A,\mu,\alpha)$ can be regarded as the BiHom-associative
algebra $(A,\mu,\alpha,\alpha)$.

\begin{df}
A \textbf{left BiHom-alternative algebra} $($resp. \textbf{right BiHom-alternative algebra}$)$ is a quadruple $(A,\mu,\alpha,\beta)$
 where $\alpha,\beta: A \rightarrow A$, are linear maps and $\mu : A\times A\rightarrow A$ is a bilinear map such that  $\alpha  \beta =\beta \alpha$, $\alpha\mu=\mu\alpha^{\otimes2}$, $\beta\mu=\mu\beta^{\otimes 2}$ and
 satisfying the \textbf{left BiHom-alternative identity},
\begin{equation}\label{sa1}
as_{\alpha,\beta}(\beta(x),\alpha(y),z)+as_{\alpha,\beta}(\beta(y),\alpha(x),z)=0,
\end{equation}
respectively, the \textbf{right BiHom-alternative identity},
\begin{equation}\label{sa2}
as_{\alpha,\beta}(x,\beta(y),\alpha(z))+as_{\alpha,\beta}(x,\beta(z),\alpha(y))=0,
\end{equation}
 for all $x,y,z \in A$.
A \textbf{BiHom-alternative} algebra is one which is both a left and right BiHom-alternative algebra.
\end{df}

Observe that when $\alpha=\beta=id$ the left BiHom-alternative identity $(\ref{sa1})$ (resp.  right BiHom-alternative identity $(\ref{sa2}$))  reduces to  left alternative identity (resp.  right alternative identity).

\begin{rem}
Any BiHom-associative algebra is a BiHom-alternative algebra.
\end{rem}

\begin{lem} Let $(A,\mu,\alpha,\beta)$ be a BiHom-algebra such that  $\alpha  \beta =\beta \alpha$, $\alpha\mu=\mu\alpha^{\otimes2}$ and $\beta\mu=\mu\beta^{\otimes 2}$. Then
\begin{itemize}
  \item [(i)] $A$ is a left BiHom-alternative algebra if and only if
for all $x,y \in A$, we have
$$as_{\alpha,\beta}(\beta(x),\alpha(x),y)=0.$$
  \item [(ii)]$A$ is a right BiHom-alternative algebra if and only if
for all $x,y \in A$, we have
$$as_{\alpha,\beta}(x,\beta(y),\alpha(y))=0.$$
\end{itemize}
\end{lem}
\begin{proof}
  Follows by a direct computation that is left to the reader.
\end{proof}

\begin{df}
Let $(A,\mu,\alpha,\beta)$ and $(A^{'},\mu^{'},\alpha^{'},\beta^{'})$ be two BiHom-alternative algebras. A  homomorphism $f:A\longrightarrow A^{'}$ is said to be a BiHom-alternative algebra morphism  if the following holds
$$f  \mu= \mu^{'} (f \otimes f),~~f \alpha=\alpha^{'} f,~~f\beta=\beta^{'} f. $$
\end{df}

\begin{df}
A BiHom-algebra $(A,\mu,\alpha,\beta)$ is called
\begin{itemize}
  \item [(i)] regular if $\alpha$ and $\beta$ are algebra automorphisms.
  \item [(ii)] involutive if $\alpha$ and $\beta$ are two involutions,  that is $\alpha^2=\beta^2=id$.
\end{itemize}
\end{df}

\begin{prop}\label{propriete Bihom regular}
Let $(A,\mu,\alpha,\beta)$ be a regular BiHom-algebra such that  $\alpha  \beta =\beta \alpha$, $\alpha\mu=\mu\alpha^{\otimes2}$ and $\beta\mu=\mu\beta^{\otimes 2}$. Then $(A,\mu,\alpha,\beta)$ is a regular BiHom-alternative algebra if and only if, for all $x,y,z \in A$, we have:
\begin{equation}
as_{\alpha,\beta}(\beta^2(x),\alpha\beta(y),\alpha^2(z))+as_{\alpha,\beta}(\beta^2(y),\alpha\beta(x),\alpha^2(z))=0,
\end{equation}
and
\begin{equation}
as_{\alpha,\beta}(\beta^2(x),\alpha\beta(y),\alpha^2(z))+as_{\alpha,\beta}(\beta^2(x),\alpha\beta(z),\alpha^2(y))=0.
\end{equation}

\end{prop}

\begin{lem}\label{propriet ass BiH-A} Let $(A,\mu,\alpha,\beta)$ be a BiHom-alternative algebra. Then for all $x , y , z \in A$ we have
$$as_{\alpha,\beta}(\beta^2(x),\beta\alpha(y),\alpha^2(z))=-as_{\alpha,\beta}(\beta^2(z),\beta\alpha(y),\alpha^2(x)),$$
$$as_{\alpha,\beta}(\beta^2(x),\beta\alpha(y),\alpha^2(z))=as_{\alpha,\beta}(\beta^2(z),\beta\alpha(x),\alpha^2(y)),$$
$$as_{\alpha,\beta}(\beta^2(x),\beta\alpha(y),\alpha^2(x))=0.$$
\end{lem}
\begin{proof} Since $(A,\mu,\alpha,\beta)$ is a  BiHom-alternative algebra. Then, for all $x,y,z$ in $A$, we have
\begin{eqnarray*}
  as_{\alpha,\beta}(\beta^2(x),\beta\alpha(y),\alpha^2(z)) &=&-as_{\alpha,\beta}(\beta^2(x),\beta\alpha(z),\alpha^2(y))\\
  &=& as_{\alpha,\beta}(\beta^2(z),\beta\alpha(x),\alpha^2(y))\\
   &=&-as_{\alpha,\beta}(\beta^2(z),\beta\alpha(y),\alpha^2(x)).
\end{eqnarray*}
The second and the third equalities  become  trivial.
\end{proof}

\subsection{Construction Theorems and Examples}

In this section, we provide a way to construct BiHom-alternative algebras starting from an alternative algebra and two commuting  alternative algebra morphisms. This procedure was applied to associative algebras, Lie algebras, preLie algebras, Novikov algebras in \cite{GrazianiMakhloufMeniniPanaite, GuoZhangWang}.
\begin{df} Let $(A,\mu)$ be an algebra and $\alpha,\beta: A\longrightarrow A$ be two algebra morphisms such that $\alpha\beta=\beta\alpha$. Define the BiHom-algebra induced by $\alpha$ and $\beta$ as
$$A_{\alpha,\beta}=(A,\mu_{\alpha,\beta}=  \mu(\alpha\otimes\beta),\alpha,\beta).$$
\end{df}
\begin{thm}\label{induced A}
Let $(A,\mu)$ be a left alternative algebra $($resp. right alternative algebra$)$ and $\alpha,\beta: A
\longrightarrow A$ be two left alternative algebra morphisms $($resp. right alternative algebra morphisms$)$ such that $\alpha\beta=\beta\alpha$. Then $A_{\alpha,\beta}$ is a left BiHom-alternative algebra $($resp. right BiHom-alternative algebra$)$.\\
Moreover, suppose that $(A',\mu')$ is an other left alternative algebra $($resp. right alternative algebra$)$ and $\alpha',\beta': A' \longrightarrow A'$ be a two commuting left alternative algebra morphisms (resp. right alternative algebra morphisms). If $f:A \longrightarrow A'$ is a left alternative algebra morphism (resp. right alternative algebra morphism) that satisfies $f \alpha= \alpha^{'}  f$ and $f \beta= \beta' f$, then
$$ f:A_{\alpha,\beta}\longrightarrow A'_{\alpha',\beta'}$$
is a  left BiHom-alternative algebra (resp. right BiHom-alternative algebra) morphism.
\end{thm}

\begin{proof}
We show   that $A_{\alpha,\beta}$ satisfies the left BiHom-alternative identity $(\ref{sa1})$. For simplicity, we will write $\mu(x,y)=xy$ and $\mu_{\alpha,\beta}(x,y)=x\ast y$. Then  for all $x,y,z  \in A_{\alpha,\beta}$ we have
\begin{eqnarray*}
  as_{\alpha,\beta}(\beta(x),\alpha(y),z) &=& (\beta(x)\ast \alpha\beta(y))\ast \beta(z)-\alpha\beta(x)\ast (\alpha(y)\ast z) \\
   &=& (\alpha\beta(x)\alpha\beta(y))\ast \beta(z) -\alpha\beta(x)\ast (\alpha^2(y)\alpha\beta(z))\\
   &=& (\alpha^2\beta(x)\alpha^2\beta(y)) \alpha^2\beta(z) -\alpha^2\beta(x) (\alpha^2\beta(y)\alpha^2\beta(z)) \\
   &=& \alpha^2\beta\big(as(x,y,z)\big) \\
   &=&  -\alpha^2\beta\big(as(y,x,z)\big)\\
   &=& -as_{\alpha,\beta} (\beta(y),\alpha(x),z)
\end{eqnarray*}

The second assertion follows from
$$f\mu_{\alpha,\beta}=f\mu(\alpha\otimes\beta)=\mu^{'}(f \otimes f)(\alpha\otimes\beta)=\mu^{'}(\alpha^{'}\otimes\beta^{'})(f \otimes f)=
\mu^{'}_{\alpha^{'},\beta^{'}}(f \otimes f).$$
\end{proof}

\begin{rem}
Similarly to previous theorem, one may construct a new BiHom-alternative algebra starting with a given BiHom-alternative algebra and a pair of commuting BiHom-algebra morphisms.
\end{rem}

\begin{exa}
In \cite{albert3} (p.\ 320-321) Albert constructed a five-dimensional right alternative algebra $(A,\mu)$ that is not left alternative.  In terms of a basis $\{e,u,v,w,z\}$ of $A$, its multiplication $\mu$ is given by

\[
\mu(e,e) = e,\quad \mu(e,u) = v,\quad \mu(u,e) = u,\quad
\mu(e,w) = w - z,\quad \mu(e,z) = z = \mu(z,e),
\]
where the unspecified products of  basis elements are all $0$.

Let $\gamma, \delta, \epsilon, b \in \mathbb{K}$ be arbitrary scalars such that $\delta \neq 0,1$.  Consider the linear maps $\alpha_{\gamma, \delta,\epsilon} = \alpha \colon A \to A$ and $\beta_b=\beta: A \to A$  given by\\
\begin{center}
$\alpha=\begin{pmatrix}
       1 & 0 & 0 & 0 & 0 \\
\epsilon & \delta & 0 & 0 & 0 \\
\epsilon & 0 & \delta & 0 & 0 \\
       0 & 0 & 0 & \gamma & 0 \\
      0  & 0 & 0 & 0 & \gamma \\
         \end{pmatrix}$   and
         $\beta=\begin{pmatrix}
                 1 & 0 & 0 & 0 & 0 \\
                 0 & 1 & 0 & 0 & 0 \\
                 0 & 0 & 1 & 0 & 0 \\
                 0 & 0 & 0 & b & 0 \\
                 0 & 0 & 0 & 0 & b \\
               \end{pmatrix}$
\end{center}
We claim that $\alpha$ and $\beta$ are two algebra morphisms on $A$ which commutes.  Indeed, suppose
\[
x = \lambda_1 e + \lambda_2 u + \lambda_3 v + \lambda_4 w + \lambda_5 z,\quad
y = \theta_1 e + \theta_2 u + \theta_3 v + \theta_4 w + \theta_5 z
\]
are two arbitrary elements in $A$ with all $\lambda_i, \theta_j \in \mathbb{K}$.  Then
\[
\begin{split}
\mu(x,y) &= \lambda_1\theta_1 e + \lambda_2\theta_1 u + \lambda_1\theta_2 v + \lambda_1\theta_4 w + (\lambda_1(\theta_5 - \theta_4) + \lambda_5\theta_1)z,\\
\alpha(x) &= \lambda_1 e + (\lambda_1\epsilon + \lambda_2\delta)u + (\lambda_1\epsilon + \lambda_3\delta)v + \lambda_4\gamma w + \lambda_5\gamma z,
\end{split}
\]
and similarly for $\alpha(y)$.  A quick computation then shows that
\[
\begin{split}
\alpha(\mu(x,y))
&= \lambda_1\theta_1 e + \theta_1(\lambda_1 \epsilon + \lambda_2\delta)u + \lambda_1(\theta_1\epsilon + \theta_2\delta)v\\
&+ \lambda_1\theta_4\gamma w + \gamma(\lambda_1(\theta_5 - \theta_4) + \lambda_5\theta_1)z\\
&= \mu(\alpha(x),\alpha(y)).
\end{split}
\]
using a similar calculus, we prove that $\beta(\mu(x,y))=\mu(\beta(x),\beta(y))$.

On the other hand, it is clear that $\alpha \beta=\beta \alpha$.

By Theorem  \ref{induced A}, there is a  right BiHom-alternative algebra
$
A_{\alpha,\beta} $, where the multiplication $\mu_{\alpha,\beta}$
is given by the following table:$$\begin{array}{c|ccccc}
                                  \mu_{\a,\b} & e & u & v & w & z \\
                                  \hline
                                  e & e+\v v & v & 0 & bw-bz & bz \\
                                  u & \d u & 0 & 0 & 0 & 0\\
                                  v & 0 & 0 & 0 & 0 & 0 \\
                                  w & 0 & 0 & 0 & 0 & 0 \\
                                  z & \g e  & 0 & 0 & 0 & 0
                                \end{array}$$

Now, we prove the following statements:
\begin{itemize}
  \item  [(1)] $A_{\alpha,\beta}$ is not left BiHom-alternative.
  \item [(2)] $(A,\mu_{\alpha,\beta})$ is not right alternative.
\end{itemize}
To show that $A_{\alpha,\beta}$ is not left BiHom-alternative, observe that
\begin{eqnarray*}
 as_{\alpha,\beta}\big(\beta(e),\alpha(e),u\big)  &=&as_{\alpha,\beta}\big(e,e+\epsilon u+\epsilon v, u\big)  \\
   &=& as_{\alpha,\beta}\big(e,e, u\big)+\underbrace{as_{\alpha,\beta}\big(e,\epsilon u,  u\big)}_{=0} +\underbrace{as_{\alpha,\beta}\big(e,\epsilon v, 2 u\big)}_{=0}    \\
   &=& \delta^2v \neq 0\ ( \text{since}\ \delta \neq 0).
\end{eqnarray*}
On the other hand, to see that $(A,\mu_{\alpha,\beta})$ is not right alternative, observe that
\begin{eqnarray*}
  \mu_{\alpha,\beta}(\mu_{\alpha,\beta}(u,e),e)-\mu_{\alpha,\beta}(u,\mu_{\alpha,\beta}(e,e)) &=& \mu_{\alpha,\beta}(\delta ue,e)-\mu_{\alpha,\beta}(u,e+\epsilon u) \\
   &=& (\delta^2-\delta)u
\end{eqnarray*}
which is not $0$ because $\delta \neq 0,1$.

\end{exa}

\begin{exa}
\label{ex:oct}
In this example, we describe a BiHom-alternative algebra  that is not alternative.  Recall that the octonions is an eight-dimensional alternative (but not associative) algebra $\mathbf{O}$ with basis $\{e_0,\ldots,e_7\}$ and the following multiplication table, where $\mu$ denotes the multiplication in $\mathbf{O}$.
\begin{center}
\begin{tabular}{c|c|c|c|c|c|c|c|c}
$\mu$ & $e_0$ & $e_1$ & $e_2$ & $e_3$ & $e_4$ & $e_5$ & $e_6$ & $e_7$ \\\hline
$e_0$ & $e_0$ & $e_1$ & $e_2$ & $e_3$ & $e_4$ & $e_5$ & $e_6$ & $e_7$ \\\hline
$e_1$ & $e_1$ & $-e_0$ & $e_4$ & $e_7$ & $-e_2$ & $e_6$ & $-e_5$ & $-e_3$ \\\hline
$e_2$ & $e_2$ & $-e_4$ & $-e_0$ & $e_5$ & $e_1$ & $-e_3$ & $e_7$ & $-e_6$ \\\hline
$e_3$ & $e_3$ & $-e_7$ & $-e_5$ & $-e_0$ & $e_6$ & $e_2$ & $-e_4$ & $e_1$ \\\hline
$e_4$ & $e_4$ & $e_2$ & $-e_1$ & $-e_6$ & $-e_0$ & $e_7$ & $e_3$ & $-e_5$ \\\hline
$e_5$ & $e_5$ & $-e_6$ & $e_3$ & $-e_2$ & $-e_7$ & $-e_0$ & $e_1$ & $e_4$ \\\hline
$e_6$ & $e_6$ & $e_5$ & $-e_7$ & $e_4$ & $-e_3$ & $-e_1$ & $-e_0$ & $e_2$ \\\hline
$e_7$ & $e_7$ & $e_3$ & $e_6$ & $-e_1$ & $e_5$ & $-e_4$ & $-e_2$ & $-e_0$ \\
\end{tabular}
\end{center}
The reader is referred to \cite{baez,gt,okubo,sv} for discussion about the roles of the octonions in exceptional Lie groups, projective geometry, physics, and other applications.

One can check that there is an algebra automorphism $\alpha \colon \mathbf{O} \to \mathbf{O}$ given by
\begin{equation}
\label{octaut}
\begin{split}
\alpha(e_0) = e_0,\quad \alpha(e_1) = e_5,\quad \alpha(e_2) = e_6,\quad \alpha(e_3) = e_7,\\
\alpha(e_4) = e_1,\quad \alpha(e_5) = e_2,\quad \alpha(e_6) = e_3,\quad \alpha(e_7) = e_4.
\end{split}
\end{equation}
There is a more conceptual description of this algebra automorphism on $\mathbf{O}$.  Note that $e_1$ and $e_2$ anti-commute, and $e_3$ anti-commutes with $e_1$, $e_2$, and $e_1e_2 = e_4$. Such a triple $(e_1,e_2,e_3)$ is called a \textbf{basic triple} in \cite{baez}.  Another basic triple is $(e_5,e_6,e_7)$.  Then $\alpha$ is the unique automorphism on $\mathbf{O}$ that sends the basic triple $(e_1,e_2,e_3)$ to the basic triple $(e_5,e_6,e_7)$.

On the other hand, set $\beta=id$. Then $\alpha$ and $\beta$ are two commuting morphisms on the alternative algebra.

By Theorem  \ref{induced A}, there is a multiplicative BiHom-alternative algebra
$
\mathbf{O}_{\alpha,\beta} $, where the multiplication $\mu_{\alpha,\beta}$
is given by the following table:
\begin{center}
\begin{tabular}{c|c|c|c|c|c|c|c|c}
$\mu_{\a,\b}$ & $e_0$ & $e_1$ & $e_2$ & $e_3$ & $e_4$ & $e_5$ & $e_6$ & $e_7$ \\\hline
$e_0$ & $e_0$ & $e_1$ & $e_2$ & $e_3$ & $e_4$ & $e_5$ & $e_6$ & $e_7$\\\hline
$e_1$ & $e_5$ & $-e_6$ & $e_3$ & $-e_2$ & $-e_7$ & $-e_0$ & $e_1$ & $e_4$\\\hline
$e_2$ &  $e_6$ & $e_5$ & $-e_7$ & $e_4$ & $-e_3$ & $-e_1$ & $-e_0$ & $e_2$ \\\hline
$e_3$ &$e_7$ & $e_3$ & $e_6$ & $-e_1$ & $e_5$ & $-e_4$ & $-e_2$ & $-e_0$ \\ \hline
$e_4$ & $e_1$ & $-e_0$ & $e_4$ & $e_7$ & $-e_2$ & $e_6$ & $-e_5$ & $-e_3$ \\\hline
$e_5$ & $e_2$ & $-e_4$ & $-e_0$ & $e_5$ & $e_1$ & $-e_3$ & $e_7$ & $-e_6$ \\\hline
$e_6$ &  $e_3$ & $-e_7$ & $-e_5$ & $-e_0$ & $e_6$ & $e_2$ & $-e_4$ & $e_1$ \\\hline
$e_7$ &  $e_4$ & $e_2$ & $-e_1$ & $-e_6$ & $-e_0$ & $e_7$ & $e_3$ & $-e_5$ \\
\end{tabular}
\end{center}
Note that $\mathbf{O}_{\a,\b}$ is not alternative because
$$\mu_{\a,\b}(\mu_{\a,\b}(e_1,e_1),e_0)=-e_3\neq -e_0=\mu_{\a,\b}(e_1,\mu_{\a,\b}(e_1,e_0)).$$
\end{exa}

Theorem \ref{induced A} gives a procedure to construct BiHom-alternative algebras using ordinary alternative
algebras and their algebras endomorphisms.

In the following proposition, we provide a construction of alternative algebra from BiHom-alternative algebra.  In what follows, we often write $\mu(a,b)$ as $ab$ and omit the subscript in the BiHom-associator $as_{\alpha,\beta}$  when there is no danger of confusion.

\begin{prop}
Let $(A,\mu,\alpha,\beta)$ be a regular  BiHom-alternative algebra. Define a new multiplication on $A$ by $x\star y=\mu(\alpha^{-1}(x),\beta^{-1}(y))$.
Then $(A,\star)$ is an alternative algebra.
\end{prop}

\begin{proof}
We write $\mu(x,y)=xy$ and $x \star y=\alpha^{-1}(x)\beta^{-1}(y)$, for all $x, y \in A$. Hence, it needs to show
\begin{equation}\label{left}
    as(x,y,z)+as(y,x,z)=0
\end{equation}
\begin{equation}\label{right}
    as(x,y,z)+as(x,z,y)=0
\end{equation}
for all $x,y \in A$. Since $\alpha$ and $\beta$ are bijective, we have
\begin{eqnarray*}
 & & as(x,y,z)+as(y,x,z) = (x \star y)\star z-x \star (y \star z)+ (y \star x)\star z-y \star (x \star z) \\
   &=& (\alpha^{-2} (x)\beta^{-1}(y))\star z-x \star (\alpha^{-1}(y)\beta^{-1}(z))+ (\alpha^{-1}(y) \beta^{-1}(x))\star z-y \star (\alpha^{-1}(x) \beta^{-1}(z))  \\
   &=& (\alpha ^{-2}(x)(\alpha\beta)^{-1}(y))\beta^{-1}(z)-\alpha^{-1}(x) ((\alpha\beta)^{-1}(y)\beta^{-2}(z))\\&+& (\alpha^{-2}(y) (\alpha\beta)^{-1}(x))\beta^{-1}(z)-\alpha^{-1}(y)((\alpha\beta)^{-1}(x) \beta^{-2}(z))   \\
   &=& as_{\alpha,\beta}\big(\alpha^{-2}(x),(\alpha\beta)^{-1}(y),\beta^{-2}(z)\big)+as_{\alpha,\beta}(\alpha^{-2}(y),(\alpha\beta)^{-1}(x),\beta^{-2}(z)\big)\\
    &=& as_{\alpha,\beta}\big(\beta(\beta^{-1}\alpha^{-2}(x)),\alpha(\beta^{-1}\alpha^{-2}(y)),\beta^{-2}(z)\big)+as_{\alpha,\beta}\big(\beta(\beta^{-1}\alpha^{-2}(y)),\alpha(\beta^{-1}\alpha^{-2}(x)),\beta^{-2}(z)\big)=0,
\end{eqnarray*}
since $A$ is a BiHom-alternative algebra. Using a similar calculation, one can prove \ref{right}.

\end{proof}\begin{cor}

Let $(A,\mu,\alpha,\beta)$ be an involutive BiHom-alternative algebra. Define a new multiplication on $A$ by $x\star y=\mu(\alpha(x),\beta(y))$.
Then $(A,\star)$ is an alternative algebra.
\end{cor}

\section{BiHom-Malcev  algebras}
\label{sec:BiHommaltsev}

In this section we define BiHom-Malcev  algebras and study their general properties.  Other characterizations of the BiHom-Malcev  identity are given.  We prove also a construction result for BiHom-Malcev  algebras.  In addition, we show that every regular BiHom-alternative algebra gives rise to a BiHom-Malcev  algebra via the BiHom-commutator bracket.  This means that regular BiHom-alternative algebras are all BiHom-Malcev  admissible algebras, generalizing the well-known fact that Hom-alternative algebras (resp. alternative algebras) are Hom-Malcev  admissible (resp. Malcev  admissible algebras).  Additional properties of BiHom-alternative algebras are shown.

Since BiHom-Malcev  algebras generalize BiHom-Lie algebras (as we will see shortly), which in turn generalize Lie algebras, we use the bracket notation $[\cdot,\cdot]$ to denote their multiplications.

\begin{df}\cite{GrazianiMakhloufMeniniPanaite}
A  BiHom-Lie algebra is a $4$-tuple $(A, [.,.],\alpha,\beta)$ where $\alpha,\beta: A\rightarrow A$ are linear maps and $[.,.]: A\times A\rightarrow A$ is a bilinear map satisfying, for all $x,y,z \in A$, the following conditions:
\begin{itemize}
  \item [(i)] $\alpha  \beta=\beta   \alpha,$
  \item [(ii)] $\alpha([x,y])=[\alpha(x),\alpha(y)]$, \ $\beta([x,y])=[\beta(x),\beta(y)]],$
  \item [(iii)] $[\beta(x),\alpha(y)]=-[\beta(y),\alpha(x)]$ (BiHom-skewsymmetry),
  \item [(iv)] $J_{\alpha,\beta}(x,y,z)=0$  (BiHom-Jacobi identity).
\end{itemize}

\end{df}

 Now, we introduce the notion of BiHom-Malcev  algebras which is a generalization of both Malcev  and Hom-Malcev  algebras.

\begin{df}
A \textbf{BiHom-Malcev  algebra} is a $4$-tuple $(A,[.,.],\alpha,\beta)$ where $\alpha,\beta: A\rightarrow A$ are linear maps and $[.,.]: A\otimes A\rightarrow A$ is a bilinear map satisfying, for all $x,y,z \in A$, the following conditions:
\begin{itemize}
  \item [(i)] $\alpha  \beta=\beta   \alpha$,
  \item [(ii)] $\alpha([x,y])=[\alpha(x),\alpha(y)]$, \ $\beta([x,y])=[\beta(x),\beta(y)]$,
  \item [(iii)] $[\beta(x),\alpha(y)]=-[\beta(y),\alpha(x)]$ (BiHom-skewsymmetry),
  \item [(iv)] $J_{\alpha,\beta}(\alpha\beta(x),\alpha\beta(y),[\beta(x),\alpha(z)])=[J_{\alpha,\beta}(\beta(x),\beta(y),\beta(z)),\alpha^2\beta^2(x)]$ (BiHom-Malcev  identity).
\end{itemize}
\end{df}

Observe that when $\alpha =\beta= id$, by the skew-symmetry of $[\cdot,\cdot]$, the BiHom-Malcev  identity reduces to the Malcev  identity $J(x,y,[x,z])=[J(x,y,z),x]$  or equivalently,
\begin{equation}
\label{maltsevidentity}
[[x,y],[x,z]] = [[[x,y],z],x] + [[[y,z],x],x] + [[[z,x],x],y]
\end{equation}
for all $x,y,z \in A$.

\begin{thm}\label{BiHom-Malcev  construct}
Let $(A,[.,.])$ be a Malcev  algebra and $\alpha, \beta :A\rightarrow A$ be two algebra morphisms such that: $\alpha  \beta=\beta  \alpha$. Then $(A,[.,.]',\alpha,\beta)$ is a BiHom-Malcev  algebra, where $[.,.]':A\times A \rightarrow A $ is defined by $[x,y]'=[\alpha(x),\beta(y)],\ x,y \in A$.
\end{thm}

\begin{proof}
First, we check that the bracket product $[.,.]'$  is compatible with the structure maps $\alpha$ and $\beta$. For any $x,y \in A$, we have
$$[\alpha(x),\alpha(y)]'=[\alpha^2(x),\alpha\beta(y)]=\alpha\Big([\alpha(x),\beta(y)]\Big)=\alpha\big([x,y]'\big)$$
Similarly, one can check  that $[\beta(x),\beta(y)]'=\beta\big([x,y]'\big)$.

We verify now the BiHom-skewsymmetry. Let $x, y \in A$, then we have
$$[\beta(x),\alpha(y)]'=[\alpha\beta(x),\alpha\beta(y)]=-[\alpha\beta(y),\alpha\beta(x)]=-[\beta(y),\alpha(x)]'.$$
It remains to show the BiHom-Malcev  identity. Let $x,y,z \in A$. Then we have
\begin{eqnarray*}
  J_{\alpha,\beta}(\alpha\beta(x),\alpha\beta(y),[\beta(x),\alpha(z)]') &=&  J_{\alpha,\beta}(\alpha\beta(x),\alpha\beta(y),[\alpha\beta(x),\alpha\beta(z)])  \\
  &=&\alpha\beta\Big(\circlearrowleft_{x,y,[x,z]}[\beta^2(x),[\beta(y)[\alpha(x),\alpha(z)]]']'\Big)\\
  &=&\alpha^2\beta^3\Big(\circlearrowleft_{x,y,[x,z]}[x,[y,[x,z]]]\Big)\\
  &=&\alpha^2\beta^3\Big(J(x,y,[x,z])\Big).
\end{eqnarray*}
On the other hand, we have
\begin{eqnarray*}
   \Big[J_{\alpha,\beta}(\beta(x),\beta(y),\beta(z)),\alpha^2\beta^2(x)\Big]'&=&\Big[\circlearrowleft_{x,y,z}\big[[\alpha^2\beta^2(x),\alpha^2\beta^2(y)],\alpha^2\beta^2(z)\big],\alpha^2\beta^2(x)\Big]'  \\
   &=&\Big[\circlearrowleft_{x,y,z}\big[[\alpha^3\beta^2(x),\alpha^3\beta^2(y)],\alpha^3\beta^2(z)\big],\alpha^2\beta^3(x)\Big] \\
   &=&\alpha^2\beta^3\Big([J(x,y,z),x]\Big) =  \alpha^2\beta^3\Big(J(x,y,[x,z])\Big)\\
   &=&J_{\alpha,\beta}(\alpha\beta(x),\alpha\beta(y),[\beta(x),\alpha(z)]').
\end{eqnarray*}
This shows that the BiHom-Malcev  identity holds.
\end{proof}

\begin{rem}
Similarly to previous theorem, one may construct a new BiHom-Malcev algebra starting with a given BiHom-Malcev algebra and a pair of commuting BiHom-algebra morphisms.
\end{rem}

We now discuss same examples of BiHom-Malcev  algebras that can be constructed using Theorem \ref{BiHom-Malcev  construct}.
\begin{exa}\cite{sagle} (Example 3.1)
There is a four-dimensional non-Lie Malcev  algebra $(A,[\cdot,\cdot])$ with basis $\{e_1,e_2,e_3,e_4\}$ and multiplication table:
\begin{center}
\begin{tabular}{c|cccc}
$[\cdot,\cdot]$ & $e_1$ & $e_2$ & $e_3$ & $e_4$ \\ \hline
$e_1$ & $0$ & $-e_2$ & $-e_3$ & $e_4$ \\
$e_2$ & $e_2$ & $0$ & $2e_4$ & $0$ \\
$e_3$ & $e_3$ & $-2e_4$ & $0$ & $0$ \\
$e_4$ & $-e_4$ & $0$ & $0$ & $0$ \\
\end{tabular}
\end{center}
We apply Theorem \ref{BiHom-Malcev  construct} to this Malcev  algebra,  Using suitable algebra morphisms.

One can check that one class of algebra morphisms $\alpha:  A \to A$ is given by
\[
\begin{split}
&\alpha(e_1) = e_1 + d_1d_2e_3 + d_1d_2^2e_4,\;
\alpha(e_2) = e_2 +d_1e_3 + d_1d_2e_4,\\
&\alpha(e_3) = e_3, \;
\alpha(e_4) = e_4,
\end{split}
\]

The second class of algebra morphisms $\beta: A \to A$ is given by
\[
\begin{split}
\beta(e_1) = -e_1 + d_2e_2 + d_3e_3 + d_4e_4,\;
\beta(e_2) = d_5e_4,\;
\beta(e_3) = 0 = \beta(e_4),
\end{split}
\]
where $d_1, d_2,d_3,d_4,d_5$ are arbitrary scalars in $\mathbb{K}$.

It is easy to check that the morphisms $\alpha$ and $\beta$ commutes.  Hence, using Theorem \ref{BiHom-Malcev  construct}, one can conclude that the algebra  $(A,[.,.]',\alpha,\beta)$ is a BiHom-Malcev  algebra, where $[.,.]':A\times A \rightarrow A $ is defined by $[x,y]'=[\alpha(x),\beta(y)],\ x,y \in A$ and its multiplication table is given by
\begin{center}
\begin{tabular}{c|cccc}
$[\cdot,\cdot]'$ & $e_1$ & $e_2$ & $e_3$ & $e_4$ \\ \hline
$e_1$ & $-d_2e_2-(d_1d_2+d_3)e_3+(d_4-2d_2+d_1d_2^2)e_4$ & $be_4$ & $0$ & $0$ \\
$e_2$ & $-e_2-d_1e_3+(d_1d_2+2d_3-2d_2)e_4$ & $0$ & $0$ & $0$ \\
$e_3$ & $-e_3-2d_2e_4$ & $0$ & $0$ & $0$ \\
$e_4$ & $e_4$ & $0$ & $0$ & $0$ \\
\end{tabular}
\end{center}

\end{exa}

\begin{exa}\cite{sagle} (Example 3.4)
There is a five-dimensional non-Lie Malcev  algebra $(A,[\cdot,\cdot])$ with basis $\{e_1,e_2,e_3,e_4,e_5\}$ and multiplication table:
\begin{center}
\begin{tabular}{c|ccccc}
$[\cdot,\cdot]$ & $e_1$ & $e_2$ & $e_3$ & $e_4$ & $e_5$ \\ \hline
$e_1$ & $0$ & $0$ & $0$ & $e_2$ & $0$ \\
$e_2$ & $0$ & $0$ & $0$ & $0$ & $e_3$ \\
$e_3$ & $0$ & $0$ & $0$ & $0$ & $0$ \\
$e_4$ & $-e_2$ & $0$ & $0$ & $0$ & $0$ \\
$e_5$ & $0$ & $-e_3$ & $0$ & $0$ & $0$
\end{tabular}
\end{center}
Let the two algebra morphisms $\alpha$ and $\beta$ defined on $A$ by
\[
\begin{split}
\alpha(e_1) &= a_1e_1 + a_2e_2 + a_3e_3 + a_4e_4 + a_5e_5,\\
\alpha(e_2) &= (a_1b_4 - a_4b_1)e_2 + (a_2b_5 - a_5b_2)e_3,\\
\alpha(e_3) &= (a_1b_4 - a_4b_1)c_5e_3,\\
\alpha(e_4) &= b_1e_1 + b_2e_2 + b_3e_3 + b_4e_4 + b_5e_5,\\
\alpha(e_5) &= c_1e_1 + c_2e_2 + c_3e_3 + c_4e_4 + c_5e_5,
\end{split}
\]
with $a_i, b_j, c_k \in \mathbb{K}$, such that
\[
\begin{split}
a_5(a_4b_1 - a_1b_4) &= 0 = b_5(a_4b_1 - a_1b_4),\\
a_1c_4 = a_4c_1,\quad a_2c_5 &= a_5c_2,\quad b_1c_4 = b_4c_1,\quad b_2c_5 = b_5c_2.
\end{split}
\]
and $\beta=id$. Then,
 using Theorem \ref{BiHom-Malcev  construct}, one can conclude that the algebra  $(A,[.,.]',\alpha,\beta)$ is a BiHom-Malcev  algebra, where $[.,.]':A\times A \rightarrow A $ is defined by $[x,y]'=[\alpha(x),\beta(y)],\ x,y \in A$. Its multiplication table is given by:
\begin{center}
\begin{tabular}{c|ccccc}
$[\cdot,\cdot]'$ & $e_1$ & $e_2$ & $e_3$ & $e_4$ & $e_5$ \\ \hline
$e_1$ & $-a_4e_2$ & $-a_5e_3$ & $0$ & $a_1e_2$ & $a_2e_3$ \\
$e_2$ & $0$ & $0$ & $0$ & $0$ & $(a_1b_4-a_4b_1)e_3$ \\
$e_3$ & $0$ & $0$ & $0$ & $0$ & $0$ \\
$e_4$ & $-b_4e_2$ & $-b_5e_3$ & $0$ & $b_1e_2$ & $b_2e_3$ \\
$e_5$ & $-c_4e_2$ & $-b_5e_3$ & $0$ & $c_1e_2$ & $c_2e_3$
\end{tabular}
\end{center}

\end{exa}

\begin{df}
Let $(A,\mu,\alpha,\beta)$ be a regular BiHom-algebra. Define the commutator BiHom-algebra as the BiHom-algebra
$$A^-=(A,[\cdot,\cdot]=\mu\circ(id-(\alpha^{-1}\beta\otimes\alpha\beta^{-1})\circ\tau),\alpha,\beta)$$
where $\tau(x,y)=(y,x)$. That is, for all $x,y \in A$, we have
$$[x,y]=\mu(x,y)-\mu(\alpha^{-1}\beta(y),\alpha\beta^{-1}(x))$$
\end{df}

We call a BiHom-algebra $A$, BiHom-Malcev-admissible if $A^-$ is a BiHom-Malcev  algebra.
\begin{exa}
\label{ex:malad}
A Malcev-admissible algebra is defined as an algebra $(A,\mu)$ for which the commutator algebra $A^- = (A,[\cdot,\cdot]=\mu\circ(id - \tau))$ is a Malcev  algebra, i.e., $A^-$ satisfies the Malcev  identity.  identifying algebras as BiHom-algebras with identity twisting maps, a Malcev-admissible algebra is equivalent to a BiHom-Malcev-admissible algebra with $\alpha =\beta= id$.
\end{exa}

Next we consider the relationship between the BiHom-associator  in a BiHom-algebra $A$ and the BiHom-Jacobiator  in its commutator BiHom-algebra $A^-$.

\begin{lem}
\label{lem:hahj}
Let $(A,\mu,\alpha,\beta)$ be a regular BiHom-algebra.  Then, in the BiHom-algebra $A^-$, we have
$$J_{\alpha,\beta} = as_{\alpha,\beta}\circ(\alpha^{-1}\beta^2\otimes \beta \otimes \alpha) \circ (id + \xi + \xi^2) \circ (id - \delta),$$
where $\xi(x \otimes y \otimes z) = z \otimes x \otimes y$ and $\delta(x \otimes y \otimes z) = x \otimes z \otimes y$.
\end{lem}

\begin{proof}
For $x,y,z \in A$, we have:
\begin{eqnarray*}
 J_{\alpha,\beta}(x,y,z)  &=&[\beta^2(x),[\beta(y),\alpha(z)]]+[\beta^2(z),[\beta(x),\alpha(y)]]+[\beta^2(y),[\beta(z),\alpha(x)]] \\
   &=&\beta^2(x)(\beta(y)\alpha(z))-(\alpha^{-1}\beta^2(y)\beta(z))\alpha\beta(x)-\beta^2(x)(\beta(z)\alpha(y))\\
   &+&(\alpha^{-1}\beta^2(z)\beta(y))\alpha\beta(x) +\beta^2(z)(\beta(x)\alpha(y))-(\alpha^{-1}\beta^2(x)\beta(y))\alpha\beta(z)   \\
   &-& \beta^2(z)(\beta(y)\alpha(x))+(\alpha^{-1}\beta^2(y)\beta(x))\alpha\beta(z) +  \beta^2(y)(\beta(z)\alpha(x))  \\
   &-&(\alpha^{-1}\beta^2(z)\beta(x))\alpha\beta(y)-\beta^2(y)(\beta(x)\alpha(z))+(\alpha^{-1}\beta^2(x)\beta(z))\alpha\beta(y)\\
   &=& as_{\alpha,\beta}\circ(\alpha^{-1}\beta^2\otimes \beta \otimes \alpha) \circ (id + \xi + \xi^2) \circ (id - \delta)(x,y,z)
\end{eqnarray*}

\end{proof}

\begin{prop}
Let $(A,\mu,\alpha,\beta)$ be a regular BiHom-alternative algebra and let  $A^-$ be its commutator BiHom-algebra. Then
$$J_{\alpha,\beta}(x,y,z)=6as_{\alpha,\beta}(\alpha^{-1}\beta^2(x),\beta(y),\alpha(z)),\ \textrm{for all}\ x, y,z \in A$$
 where the BiHom-associator is taken in $A$ and the BiHom-Jacobiator is considered in $A^-$.
\end{prop}

\begin{proof}
 The result follows immediately from Lemma \ref{lem:hahj} and Proposition \ref{propriete Bihom regular}.
\end{proof}

It is known that, given a BiHom-associative algebra $A$, its commutator BiHom-algebra $A^-$ is
a BiHom-Lie algebra. Also, the commutator algebra of any Hom-alternative algebra is a Hom-Malcev  algebra.
The following main result generalizes both of these facts. It gives us a large class of
BiHom-Malcev-admissible algebras that are in general not BiHom-Lie-admissible.

\begin{thm}\label{BiHom-Malcev  admissible}
Every regular BiHom-alternative algebra is  BiHom-Malcev-admissible.
\end{thm}
Before starting the proof of Theorem \ref{BiHom-Malcev  admissible}, we will first establish some identities about the BiHom-associator and the BiHom-Jacobiator. Then we will go back to the proof of Theorem  \ref{BiHom-Malcev  admissible}. 

\begin{df}
\label{def:alternating}
Let $V$ be a $\mathbb{K}$ -module and $f \colon V^{\otimes n} \to V$ be an $n$-linear map for some $n \geq 2$.  We say that $f$ is alternating if
$
f = \textrm{sign}(\sigma) f \circ \sigma
$
for each permutation $\sigma$ on $n$ letters, where $\textrm{sign}(\sigma) \in \{\pm 1\}$ is the signature of $\sigma$.
\end{df}

\begin{lem}
\label{lem:alternating}
Let $V$ be a $\mathbb{K}$-module and $f \colon V^{\otimes n} \to V$ be an $n$-linear map.  Then the following statements are equivalent:
\begin{enumerate}
\item
$f$ is alternating,
\item
$f(x_1,\ldots,x_n) = 0$ whenever $x_i = x_j$ for some $i \not= j$,
\item
$f = -f\circ\tau$ for each transposition $\tau$ on $n$ letters,
\item
$f = (-1)^{n-1}f\circ\xi$ and $f = -f\circ\eta$, where $\xi$ is the cyclic permutation $(12\cdots n)$ and $\eta$ is the adjacent transposition $(n-1,n)$.
\end{enumerate}
\end{lem}

\begin{rem}
Let $(A,\mu,\alpha,\beta)$ be a BiHom-algebra. Then $A$ is  a regular BiHom-alternative if and only if  $as\circ (\beta^2 \otimes \alpha\beta \otimes \alpha^2)$ is alternating.
\end{rem}

\begin{df}
Let $(A,\mu,\a,\b)$ be a BiHom-algebra. Define  the function $H: A^{\otimes 4} \to A$ by:
\begin{eqnarray*}
  H(w,x,y,z) &=& as(\beta^2(w)\alpha\beta(x),\alpha^2\beta(y),\alpha^3(z)) - as(\beta^2(x)\alpha\beta(y),\alpha^2\beta(z),\alpha^3(w))  \\
   &+& as(\beta^2(y)\alpha\beta(z),\alpha^2\beta(w),\alpha^3(x)).
\end{eqnarray*}

\end{df}

\begin{lem}\label{lem1}
In a regular BiHom-alternative algebra $(A,\mu,\alpha,\beta)$,  the identity
$$H(w,x,y,z)=\alpha^2\beta^2(w)as(\alpha\beta(x),\alpha^2(y),\alpha^3\beta^{-1}(z))+ as(\beta^2(w),\alpha\beta(x),\alpha^2(y))\alpha^3\beta(z)$$
holds for all $w,x,y,z \in A$.
\end{lem}

\begin{proof}
First, since $(A,\mu,\alpha,\beta)$ is a regular BiHom-alternative algebra, one has
\begin{eqnarray*}
  as(\beta^2(x)\alpha\beta(y),\alpha^2\beta(z),\alpha^3(w))  &=& as\big(\beta^2(x\alpha\beta^{-1}(y)),\alpha\beta(\alpha(z)),\alpha^2(\alpha(w))\big) \\
   &=& as\big(\alpha\beta^2(w),\alpha\beta(x)\alpha^2(y),\alpha^3(z)\big).
\end{eqnarray*}
In addition,
\begin{eqnarray*}
 as(\beta^2(y)\alpha\beta(z),\alpha^2\beta(w),\alpha^3(x))  &=& as\big(\beta^2(y\alpha\beta^{-1}(z)),\alpha\beta(\alpha(w)),\alpha^2(\alpha(x))\alpha\big) \\
   &=&  as\big(\alpha\beta^2(w),\alpha^2\beta(x),\alpha^2(y)\alpha^3\beta^{-1}(z)\big).
\end{eqnarray*}
Therefore,
\begin{eqnarray*}
 H(w,x,y,z) &=&\Big((\beta^2(w)\alpha\beta(x))\alpha^2\beta(y)\Big)\alpha^3\beta(z)-
 \Big(\alpha\beta^2(w)\alpha^2\beta(x)\Big)\Big(\alpha^2\beta(y)\alpha^3(z)\Big) \\
   &-&\Big(\alpha\beta^2(w)(\alpha\beta(x)\alpha^2(y))\Big)\alpha^3\beta(z)+\alpha^2\beta^2(w)\Big((\alpha\beta(x)\alpha^2(y))\alpha^3(z)\Big)  \\
   &+& \Big(\alpha\beta^2(w)\alpha^2\beta(x)\Big)\Big(\alpha^2\beta(y)\alpha^3(z)\Big) -\alpha^2\beta^2(w)\Big(\alpha^2\beta(x)(\alpha^2(y)\alpha^3\beta^{-1}(z))\Big) \\
   &=&  \alpha^2\beta^2(w)as(\alpha\beta(x),\alpha^2(y),\alpha^3\beta^{-1}(z))+ as(\beta^2(w),\alpha\beta(x),\alpha^2(y))\alpha^3\beta(z).
\end{eqnarray*}

\end{proof}
We now build an other  map on four variables using the BiHom-associator that is alternating in a  BiHom-alternative algebra.

\begin{df}
Let $(A,\mu,\alpha,\beta)$ be a regular BiHom-algebra.  Define the \textbf{BiHom-Bruck-Kleinfeld function} $f \colon A^{\otimes 4} \to A$ as the multilinear map
\begin{equation}
\label{f}
\begin{split}
 f(w,x,y,z) &= as(\beta^2(w)\alpha\beta(x),\alpha^2\beta(y),\alpha^3(z))    - as(\beta^2(x),\alpha\beta(y),\alpha^2(z))\alpha^3\beta(w) \\
    & - \alpha^2\beta^2(x)as(\alpha\beta(w),\alpha^2(y),\alpha^3\beta^{-1}(z)).
\end{split}
\end{equation}

 Define another multi-linear map $F \colon A^{\otimes 4} \to A$ as
\begin{equation}
\label{F}
F = [\cdot,\cdot] \circ ((as\circ (\beta^2\otimes \alpha\beta \otimes \alpha^2))\otimes \alpha^3\beta) \circ (id - \xi + \xi^2 - \xi^3),
\end{equation}
where $[\cdot,\cdot] = \mu \circ (id - (\alpha^{-1}\beta\otimes \alpha\beta^{-1})\circ \tau)$ is the commutator bracket of $\mu$ and $\xi$ is the cyclic permutation
$$\xi(w \otimes x \otimes y \otimes z) = z \otimes w \otimes x \otimes y.$$
\end{df}
That is, for each $w,x,y,z \in A$ we have
\begin{eqnarray*}
  F(w,x,y,z) &=& [as(\beta^2(w),\alpha\beta(x),\alpha^2(y)),\alpha^3\beta(z)]- [as(\beta^2(z),\alpha\beta(w),\alpha^2(x)),\alpha^3\beta(y)] \\
   &+&  [as(\beta^2(y),\alpha\beta(z),\alpha^2(w)),\alpha^3\beta(x)]- [as(\beta^2(x),\alpha\beta(y),\alpha^2(z)),\alpha^3\beta(w)].
\end{eqnarray*}

The BiHom-Bruck-Kleinfeld function $f$ is the BiHom-type analogue of a map studied by Bruck and Kleinfeld (\cite{bk} (2.7)).  It is closely related to the map $F$.

\begin{lem}
\label{lem2:homalt}
In a BiHom-alternative algebra $(A,\mu,\alpha)$, we have
\[
F = f \circ (id - \rho + \rho^2),
\]
where $\rho = \xi^3$ is the cyclic permutation $\rho(w \otimes x \otimes y \otimes z) = x \otimes y \otimes z \otimes w$.
In terms of elements, the map $F$ is given by
$$F(w,x,y,z)=f(w,x,y,z)-f(x,y,z,w)+f(y,z,w,x).$$
\end{lem}

\begin{proof}
For each $w,x,y,z \in A$, on has
\begin{eqnarray*}
 & & [as(\beta^2(w),\alpha\beta(x),\alpha^2(y)),\alpha^3\beta(z)] \\
  &=& as(\beta^2(w),\alpha\beta(x),\alpha^2(y))\alpha^3\beta(z)-\alpha^2\beta^2(z)as(\alpha\beta(w),\alpha^2(x),\alpha^3\beta^{-1}(y)),
\end{eqnarray*}
\begin{eqnarray*}
   & & [as(\beta^2(z),\alpha\beta(w),\alpha^2(x)),\alpha^3\beta(y)] \\
   &=& as(\beta^2(z),\alpha\beta(w),\alpha^2(x))\alpha^3\beta(y)                                -\alpha^2\beta^2(y)as(\alpha\beta(z),\alpha^2(w),\alpha^3\beta^{-1}(x)),
\end{eqnarray*}
\begin{eqnarray*}
   & &  [as(\beta^2(y),\alpha\beta(z),\alpha^2(w)),\alpha^3\beta(x)] \\
  &=&  as(\beta^2(y),\alpha\beta(z),\alpha^2(w))\alpha^3\beta(x)-\alpha^2\beta^2(x)as(\alpha\beta(y),\alpha^2(z),\alpha^3\beta^{-1}(w)),
\end{eqnarray*}
and
\begin{eqnarray*}
   & &  [as(\beta^2(x),\alpha\beta(y),\alpha^2(z)),\alpha^3\beta(w)]\\
   &=&   as(\beta^2(x),\alpha\beta(y)\alpha^2(z)),\alpha^3\beta(w)\               -\alpha^2\beta^2(w)as(\alpha\beta(x),\alpha^2(y),\alpha^3\beta^{-1}(z)).
\end{eqnarray*}

Then, by Lemma \ref{lem1}, we get
\begin{eqnarray*}
 & & \alpha^2\beta^2(w)as(\alpha\beta(x),\alpha^2(y),\alpha^3\beta^{-1}(z))+ as(\beta^2(w),\alpha\beta(x),\alpha^2(y))\alpha^3\beta(z)   \\
   &=& as(\beta^2(w)\alpha\beta(x),\alpha^2\beta(y),\alpha^3(z))
   -  as(\beta^2(x)\alpha\beta(y),\alpha^2\beta(z),\alpha^3(w))  \\
   & & \qquad  \qquad \qquad \qquad  \qquad \qquad \qquad +as(\beta^2(y)\alpha\beta(z),\alpha^2\beta(w),\alpha^3(x)) \\
   &=&  f(w,x,y,z) + as(\beta^2(x),\alpha\beta(y),\alpha^2(z))\alpha^3\beta(w)
  +\alpha^2\beta^2(x)as(\alpha\beta(w),\alpha^2(y),\alpha^3\beta^{-1}(z))   \\
   &-&  f(x,y,z,w) - as(\beta^2(y),\alpha\beta(z),\alpha^2(w))\alpha^3\beta(x)
  -\alpha^2\beta^2(y)as(\alpha\beta(x),\alpha^2(z),\alpha^3\beta^{-1}(w))  \\
   &+&  f(y,z,w,x) + as(\beta^2(z),\alpha\beta(w),\alpha^2(x))\alpha^3\beta(y)
  +\alpha^2\beta^2(z)as(\alpha\beta(y),\alpha^2(w),\alpha^3\beta^{-1}(x)).
\end{eqnarray*}
Now, since $as\circ (\beta^2 \otimes \alpha\beta \otimes  \alpha^2 )$ is alternating, we have
\begin{eqnarray*}
 as(\alpha\beta(y),\alpha^2(w),\alpha^3\beta^{-1}(x)) &=& \alpha\beta^{-1}\Big(as(\beta^2(y),\alpha\beta(w),\alpha^2(x)) \Big)\\
   &=&  \alpha\beta^{-1}\Big(as(\beta^2(w),\alpha\beta(x),\alpha^2(y))\Big) \\
   &=& as(\alpha\beta(w),\alpha^2(x),\alpha^3\beta^{-1}(y)),
\end{eqnarray*}
\begin{eqnarray*}
 as(\alpha\beta(x),\alpha^2(z),\alpha^3\beta^{-1}(w)) &=& \alpha\beta^{-1}\Big(as(\beta^2(x),\alpha\beta(z),\alpha^2(w)) \Big)\\
   &=&  \alpha\beta^{-1}\Big(as(\beta^2(z),\alpha\beta(w),\alpha^2(x)) \Big)\\
   &=& as(\alpha\beta(z),\alpha^2(w),\alpha^3\beta^{-1}(x))
\end{eqnarray*}
and
\begin{eqnarray*}
 as(\alpha\beta(w),\alpha^2(y),\alpha^3\beta^{-1}(z)) &=& \alpha\beta^{-1}\Big(as(\beta^2(w),\alpha\beta(y),\alpha^2(z))\Big) \\
   &=&  \alpha\beta^{-1}\Big(as(\beta^2(y),\alpha\beta(z),\alpha^2(w)) \Big)\\
   &=& as(\alpha\beta(y),\alpha^2(z),\alpha^3\beta^{-1}(w)).
\end{eqnarray*}
Then $F(w,x,y,z)=f(w,x,y,z)-f(x,y,z,w)+f(y,z,w,x)$, for each $w,x,y,z \in A$.

\end{proof}

\begin{prop}
\label{prop:f alternating}
Let $(A,\mu,\alpha,\beta)$ be a regular BiHom-alternative algebra.  Then the BiHom-Bruck-Kleinfeld function $f$ is alternating.
\end{prop}

\begin{proof}
Note that $(id-\xi+\xi^2-\xi^3)\circ \xi= \xi-\xi^2+\xi^3-id=-(id-\xi+\xi^2-\xi^3)$, then $-F\circ \xi=F$, which implies that $-F=F\circ \xi^3=F\circ \rho$ (since $\rho=\xi^3$).
Observe that $\rho^3 = \xi$.  Thus, we have:
\[
\begin{split}
0 &= F \circ (id + \rho)\\
&= f \circ (id - \rho + \rho^2) \circ (id + \rho) \quad\text{(by Lemma \ref{lem2:homalt})}\\
&= f \circ (id + \rho^3)\\
&= f \circ (id + \xi).
\end{split}
\]
Equivalently, we have
\begin{equation}
\label{fxi}
f = -f\circ\xi,
\end{equation}
so $f$ changes sign under the cyclic permutation $\xi$.  From Definition \eqref{f} of $f$ and the fact that $as\circ (\beta^2 \otimes \alpha\beta \otimes \alpha^2)$ is alternating in a BiHom-alternative algebra, we infer also that
\begin{equation}
\label{feta}
f = -f\circ\eta,
\end{equation}
where $\eta$ is the adjacent transposition $\eta(w \otimes x \otimes y \otimes z) = w \otimes x \otimes z \otimes y$.  So $f$ changes sign under the transposition $\eta$.  Since the cyclic permutation $\xi$ and the adjacent transposition $\eta$ generate the symmetric group on four letters, we infer from \eqref{fxi} and \eqref{feta} that $f$ is alternating.

\end{proof}

Now, we are able to prove Theorem \ref{BiHom-Malcev  admissible}.
\begin{proof}
Let $(A,\mu,\alpha,\beta)$ be a BiHom-alternative algebra and $A^- = (A,[\cdot,\cdot],\alpha,\beta)$ be its commutator BiHom-algebra.

First, we check that the bracket product $[\cdot,\cdot]$ is compatible with the structure maps $\alpha$ and $\beta$.
For any $x, y \in A^-$, we have
\begin{eqnarray*}
   [\alpha(x),\alpha(y)]&=& \alpha(x)\alpha(y)-\alpha^{-1}\beta(\alpha(y))\alpha\beta^{-1}(\alpha(x)) \\
   &=& \alpha(x)\alpha(y)-\alpha\alpha^{-1}\beta(y)\alpha^2\beta^{-1}(x) \\
   &=&  \alpha\big([x,y]\big).
\end{eqnarray*}
Similarly, one can prove that $\beta\big([x,y]\big)=[\beta(x),\beta(y)]$.

The commutator bracket $[\cdot,\cdot]=\mu\circ(id - (\alpha^{-1}\beta\otimes \alpha\beta^{-1})\circ \tau)$ is BiHom-skewsymmetric in the sense that for all $x,y \in A$, we have

  $[\beta(x),\alpha(y)] = \beta(x)\alpha(y)-\beta(y)\alpha(x) $ and
  $ [\beta(y),\alpha(x)]= \beta(y)\alpha(x)-\beta(x)\alpha(y)$,

 then $[\beta(x),\alpha(y)] =- [\beta(y),\alpha(x)]$.
 Thus, it remains to show that the BiHom-Malcev  identity  holds in $A^-$, i.e.,
$$J_{\alpha,\beta}(\alpha\beta(x),\alpha\beta(y),[\beta(x),\alpha(z)])=[J_{\alpha,\beta}(\beta(x),\beta(y),\beta(z)),\alpha^2\beta^2(x)].$$
Note first that $J_{\alpha,\beta}(\beta(x),\beta(y),\beta(z))=6as(\alpha^{-1}\beta^3(x),\beta^2(y),\alpha\beta(z))$.

In addition, $[\beta(x),\alpha(z)]=\beta(x)\alpha(z)-\beta(z)\alpha(x)$.
Therefore,
\begin{eqnarray*}
  & &J_{\alpha,\beta}(\alpha\beta(x),\alpha\beta(y),[\beta(x),\alpha(z)]) =  J_{\alpha,\beta}(\alpha\beta(x),\alpha\beta(y),\beta(x)\alpha(z))-J_{\alpha,\beta}(\alpha\beta(x),\alpha\beta(y),\beta(z)\alpha(x))  \\
   &=& 6as(\beta^3(x),\alpha\beta^2(y),\alpha\beta(x)\alpha^2(z))-6  as(\beta^3(x),\alpha\beta^2(y),\alpha\beta(z)\alpha^2(x)) \\
   &=& 6 as\big(\beta^2(\beta(x)),\alpha\beta(\beta(y)),\alpha^2(\alpha^{-1}\beta(x)z)\big)-6 as\big(\beta^2(\beta(x)),\alpha\beta(\beta(y)),\alpha^2(\alpha^{-1}\beta(x))\big) \\
   &=& 6as(\alpha^{-1}\beta^3(x)\beta^2(z),\alpha\beta^2(x),\alpha^2\beta(y)) - 6as(\alpha^{-1}\beta^3(z)\beta^2(x),\alpha\beta^2(x),\alpha^2\beta(y))\\
   &=& 6\alpha^{-1}\beta\Big(as(\beta^2(x)\alpha\beta(z),\alpha^2(x),\alpha^3(y))-as(\beta^2(z)\alpha\beta(x),\alpha^2(x),\alpha^3(y))\Big) \\
   &=& 6\alpha^{-1}\beta\bigg( f(x,z,x,y) + as(\beta^2(z),\alpha\beta(x),\alpha^2(y))\alpha^3\beta(x)
  +\alpha^2\beta^2(z)as(\alpha\beta(x),\alpha^2(x),\alpha^3\beta^{-1}(y)) \\
   &-& f(z,x,x,y) + as(\beta^2(x),\alpha\beta(x),\alpha^2(y))\alpha^3\beta(z)
  +\alpha^2\beta^2(x)as(\alpha\beta(z),\alpha^2(x),\alpha^3\beta^{-1}(y))\bigg)  \\
   &=&6\alpha^{-1}\beta\bigg(as(\beta^2(z),\alpha\beta(x),\alpha^2(y))\alpha^3\beta(x)-
   \alpha^2\beta^2(x)as(\alpha\beta(z),\alpha^2(x),\alpha^3\beta^{-1}(y))\bigg) \\
   &=& 6\alpha^{-1}\beta\big[as(\beta^2(z),\alpha\beta(x),\alpha^2(y)),\alpha^3\beta(x)\big] \\
   &=& 6\big[as(\alpha^{-1}\beta^3(x),\beta^2(y),\alpha\beta(z),\alpha^2\beta^2(x)\big] \\
   &=&  [J(\beta(x),\beta(y),\beta(z)),\alpha^2\beta^2(x)].
\end{eqnarray*}

\end{proof}


\section{Further properties of BiHom-alternative algebras}


In this section, we consider further properties of BiHom-alternative algebras, including BiHom-type analogues of the Moufang identities \cite{moufang} (Theorem \ref{BiHom moufang}) and more identities concerning the BiHom-Bruck-Kleinfeld function \eqref{f} (Proposition \ref{prop:f2}).

In any alternative algebra, the following \textbf{Moufang identities} \cite{moufang} hold:
\begin{equation}
\label{moufang}
\begin{split}
(xyx)z &= x(y(xz)),\\
((zx)y)x &= z(xyx),\\
(xy)(zx) &= x(yz)x.
\end{split}
\end{equation}
Here $xyx = (xy)x = x(yx)$ is unambiguous in an alternative algebra.  Now we provide analogues of the Moufang identities in a BiHom-alternative algebra.  The proof below is the BiHom version of that in \cite{bk} (Lemma 2.2).

\begin{thm}[\textbf{BiHom-Moufang identities}]
\label{BiHom moufang}
Let $(A,\mu,\alpha,\beta)$ be a regular  BiHom-alternative algebra.  Then the following BiHom-Moufang identities hold for all $x,y,z \in A$:
\begin{itemize}
  \item [(i)] $\Big(\beta^3(x)\big(\beta^2(y)\alpha\beta(x)\big)\Big)\alpha^2\beta^2(z)=\alpha\beta^3(x)\Big(\alpha\beta^2(y)\big(\alpha\beta(x)\alpha^2(z)\big)\Big)$.
  \item [(ii)]   $\Big(\big(\beta^2(z)\alpha\beta(x)\big)\alpha^2\beta(y)\Big)\alpha^3\beta(x)=\alpha^2\beta^2(z)\Big(\big(\alpha\beta(x)\alpha^2(y)\big)\alpha^3(x)\Big)$.
  \item[(iii)] $\Big(\alpha\beta^2(x)\alpha^2\beta(y)\Big)\Big(\alpha^2\beta(z)\alpha^3(x)\Big)=\Big(\alpha\beta^2(x)\big(\alpha\beta(y)\alpha^2(z)\big)\Big)\alpha^3\beta(x)$.
\end{itemize}
\end{thm}
 If $\alpha=\beta=id$, we get  the  Moufang identities in an alternative algebra.
\begin{proof}
(i)- We compute as follows
\begin{eqnarray*}
   & &\alpha\beta^3(x)\Big(\alpha\beta^2(y)\big(\alpha\beta(x)\alpha^2(z)\big)\Big) \\
   &=& \alpha\beta^3(x)\Big(\big(\beta^2(y)\alpha\beta(x)\big)\alpha^2\beta(z)\Big)-\alpha\beta^3(x)as\big(\beta^2(y),\alpha\beta(x),\alpha^2(z)\big)  \\
   &=& \Big(\beta^3(x)\big(\beta^2(y)\alpha\beta(x)\big)\Big)\alpha^2\beta^2(z)-as\big(\beta^3(x),\beta^2(y)\alpha\beta(x),\alpha^2\beta(z)\big)-
   \alpha\beta^3(x)as\big(\beta^2(y),\alpha\beta(x),\alpha^2(z)\big) \\
   &=& \Big(\beta^3(x)\big(\beta^2(y)\alpha\beta(x)\big)\Big)\alpha^2\beta^2(z)-
   \alpha^{-1}\beta\Big(as\big(\beta^2(\alpha(x)),\alpha\beta(y\alpha\beta^{-1}(x)),\alpha^2(\alpha(z))\big)\\
  & & \ \ \ \ \ \ \ \ \ \ \ \ \ \ \quad \qquad \qquad \qquad +\alpha^2\beta^2(x)as\big(\alpha\beta(y),\alpha^2(x),\alpha^3\beta^{-1}(z)\big)\Big) \\
   &=& \Big(\beta^3(x)\big(\beta^2(y)\alpha\beta(x)\big)\Big)\alpha^2\beta^2(z)-
  \alpha^{-1}\beta\Big(as\big(\beta^2(y)\alpha\beta(x),\alpha^2\beta(z),\alpha^3(x)\big)\\
  & & \ \ \ \ \ \ \ \ \ \ \ \ \ \ \quad \qquad \qquad \qquad -\alpha^2\beta^2(x)as\big(\alpha\beta(y),\alpha^2(z),\alpha^3\beta^{-1}(x)\big)\Big) \\
   &=&  \Big(\beta^3(x)\big(\beta^2(y)\alpha\beta(x)\big)\Big)\alpha^2\beta^2(z)-  \alpha^{-1}\beta\Big(f(y,x,z,x)+as\big(\beta^2(x),\alpha\beta(z),\alpha^2(x)\big)\alpha^3\beta(y)\Big)\\
   &=&  \Big(\beta^3(x)\big(\beta^2(y)\alpha\beta(x)\big)\Big)\alpha^2\beta^2(z) \quad (\text{Since $f$ is alternating and $A$ is BiHom-alternative})
\end{eqnarray*}

(ii)- For any $x,y,z \in A$, we have
\begin{eqnarray*}
   & &  \Big(\big(\beta^2(z)\alpha\beta(x)\big)\alpha^2\beta(y)\Big)\alpha^3\beta(x) \\
   &=& as\big(\beta^2(z),\alpha\beta(x),\alpha^2(y)\big)\alpha^3\beta(x)+ \Big(\alpha\beta^2(z)\big(\alpha\beta(x)\alpha^2(y)\big)\Big)\alpha^3\beta(x) \\
   &=& as\big(\beta^2(z),\alpha\beta(x),\alpha^2(y)\big)\alpha^3\beta(x)+ as\big(\alpha\beta^2(z),\alpha\beta(x)\alpha^2(y),\alpha^3(x)\big)+
   \alpha^2\beta^2(z)\Big( \big( \alpha\beta(x)\alpha^2(y) \big) \alpha^3(x) \Big)   \\
   &=& as\big(\beta^2(z),\alpha\beta(x),\alpha^2(y)\big)\alpha^3\beta(x)+ as\big(\beta^2(\alpha(z)),\alpha\beta(x\alpha\beta^{-1}(y)),\alpha^2(\alpha(x))\big)+
    \alpha^2\beta^2(z)\Big( \big( \alpha\beta(x)\alpha^2(y) \big) \alpha^3(x) \Big) \\
   &=&  -as\big(\beta^2(y),\alpha\beta(x),\alpha^2(z)\big)\alpha^3\beta(x)+
   as\big(\beta^2(x)\alpha\beta(y),\alpha^2\beta(x),\alpha^3(z)\big)+
   \alpha^2\beta^2(z)\Big( \big( \alpha\beta(x)\alpha^2(y) \big) \alpha^3(x) \Big)  \\
   &=& f(x,y,x,z)+  \alpha^2\beta^2(y)as\big(\alpha\beta(x),\alpha^2(x),\alpha^3\beta^{-1}(z)\big) +\alpha^2\beta^2(z)\Big( \big( \alpha\beta(x)\alpha^2(y) \big) \alpha^3(x) \Big)  \quad (\text{Definition \ref{f}} )\\
   &=& as\big(\beta^2(\alpha\beta^{-1}(x)),\alpha\beta(\alpha\beta^{-1}(x)),\alpha^2(\alpha\beta^{-1}(z))\big)+
   \alpha^2\beta^2(z)\Big( \big( \alpha\beta(x)\alpha^2(y) \big) \alpha^3(x) \Big)  \quad (\text{by Proposition \ref{prop:f alternating}}) \\
   &=&  \alpha^2\beta^2(z)\Big( \big( \alpha\beta(x)\alpha^2(y) \big) \alpha^3(x) \Big)  \quad (\text{since $ A$ is BiHom-alternative})
\end{eqnarray*}
This prove the identity (ii).

(iii)- For the third identity, we compute as follows
\begin{eqnarray*}
   & & \Big(\alpha\beta^2(x)\alpha^2\beta(y)\Big)\Big(\alpha^2\beta(z)\alpha^3(x)\Big) \\ && =\Big(\big(\beta^2(x)\alpha\beta(y)\big)\alpha^2\beta(z)\Big)\alpha^3\beta(x)-
   as\big(\beta^2(x)\alpha\beta(y),\alpha^2\beta(z),\alpha^3(x)\big) \\
   &&= \Big(\big(\beta^2(x)\alpha\beta(y)\big)\alpha^2\beta(z)\Big)\alpha^3\beta(x)-f(x,y,z,x) \\
   &&-as\big(\beta^2(y),\alpha\beta(z),\alpha^2(x)\big)\alpha^3\beta(x)-\alpha^2\beta^2(y)as\big(\alpha\beta(x),\alpha^2(z),\alpha^3\beta^{-1}(x)\big)  \\
   &&=  \Big(\big(\beta^2(x)\alpha\beta(y)\big)\alpha^2\beta(z)\Big)\alpha^3\beta(x)-as\big(\beta^2(y),\alpha\beta(z),\alpha^2(x)\big)\alpha^3\beta(x) \\
   &&=  \Big(\big(\beta^2(x)\alpha\beta(y)\big)\alpha^2\beta(z)\Big)\alpha^3\beta(x)-as\big(\beta^2(x),\alpha\beta(y),\alpha^2(z)\big)\alpha^3\beta(x)  \\
   &&=  \Big(\alpha\beta^2(x)\big(\alpha\beta(y)\alpha^2(z)\big)\Big)\alpha^3\beta(x)
\end{eqnarray*}

\end{proof}

Next we provide further properties of the BiHom-Bruck-Kleinfeld function $f$ \eqref{f}.  The following result gives two characterizations of the BiHom-Bruck-Kleinfeld function in a BiHom-alternative algebra.  It is the BiHom-type analogue of part in \cite{bk} (Lemma 2.1).

\begin{prop}
\label{prop:f2}
Let $(A,\mu,\alpha,\beta)$ be a regular  BiHom-alternative algebra.  Then the BiHom-Bruck-Kleinfeld function $f$ satisfies
\begin{equation}
\label{f'}
f = \frac{1}{3}F = as \circ \Big(\big([\cdot,\cdot]\circ(\beta^2\otimes \alpha\beta)\big)\otimes \alpha^2\beta \otimes \alpha^3\Big) \circ (id + \zeta),
\end{equation}
where $F$ is defined in \eqref{F} and $\zeta$ is the permutation $\zeta(w \otimes x \otimes y \otimes z) = y \otimes z \otimes w \otimes x$
and  $[\cdot,\cdot] = \mu \circ (id - (\alpha^{-1}\beta\otimes \alpha\beta^{-1})\circ \tau)$ is the commutator bracket of $\mu$.
\end{prop}

\begin{proof}
First note that
\[
f = -f\circ\rho = f\circ\rho^2,
\]
because $f$ is alternating (Proposition \ref{prop:f alternating}), where $\rho = \xi^3$ is the cyclic permutation $\rho(w \otimes x \otimes y \otimes z) = x \otimes y \otimes z \otimes w$.  Therefore, we have
\[
\begin{split}
F &= f \circ (id - \rho + \rho^2) \quad\text{(by Lemma \ref{lem2:homalt})}\\
&= 3f,
\end{split}
\]
which proves the first equality in \eqref{f'}.  It remains to prove that $f$ is equal to the last entry in \eqref{f'}.

Since $f$ is alternating, from its definition \eqref{f} we have
\begin{eqnarray*}
& & 2f(w,x,y,z)=  f(w,x,y,z)-f(x,w,y,z) \\
   &=&as(\beta^2(w)\alpha\beta(x),\alpha^2\beta(y),\alpha^3(z))    - as(\beta^2(x),\alpha\beta(y),\alpha^2(z))\alpha^3\beta(w) \\
    & -& \alpha^2\beta^2(x)as(\alpha\beta(w),\alpha^2(y),\alpha^3\beta^{-1}(z)) -
    as(\beta^2(x)\alpha\beta(w),\alpha^2\beta(y),\alpha^3(z))\\
   &+&  as(\beta^2(w),\alpha\beta(y),\alpha^2(z))\alpha^3\beta(x)
   + \alpha^2\beta^2(w)as(\alpha\beta(x),\alpha^2(y),\alpha^3\beta^{-1}(z))
\end{eqnarray*}
which implies that
\begin{eqnarray}
\label{f=1/3F}
   & & as\Big(\big[\beta^2(w)\alpha\beta(x)\big],\alpha^2\beta(y),\alpha^3(z)\Big) =  2f(w,x,y,z)+ \nonumber\\
& &  \big[as(\beta^2(x),\alpha\beta(y),\alpha^2(z)),\alpha^3\beta(w)\big]-
\big[as(\beta^2(w),\alpha\beta(y),\alpha^2(z)),\alpha^3\beta(x)\big].
\end{eqnarray}
Now, interchanging $(w,x)$ with $(y,z)$ in the last equality, one has
\begin{eqnarray}
\label{f=1:3F}
   & & as\Big(\big[\beta^2(y)\alpha\beta(z)\big],\alpha^2\beta(w),\alpha^3(x)\Big) =  2f(y,z,w,x)+ \nonumber\\
& &  \big[as(\beta^2(z),\alpha\beta(w),\alpha^2(x)),\alpha^3\beta(y)\big]-
\big[as(\beta^2(y),\alpha\beta(w),\alpha^2(x)),\alpha^3\beta(z)\big] \nonumber \\
   &=&  2f(w,x,y,z)+\big[as(\beta^2(z),\alpha\beta(w),\alpha^2(x)),\alpha^3\beta(y)\big] \nonumber\\
&-& \big[as(\beta^2(y),\alpha\beta(w),\alpha^2(x)),\alpha^3\beta(z)\big].
\end{eqnarray}
Adding (\ref{f=1/3F}) and (\ref{f=1:3F}) we have
\begin{eqnarray*}\nonumber
   & & as \circ \Big(\big([\cdot,\cdot]\circ(\beta^2\otimes \alpha\beta)\big)\otimes \alpha^2\beta \otimes \alpha^3\Big) \circ (id + \zeta)(w\otimes x\otimes y \otimes z) \\
   &=& as\Big(\big[\beta^2(w)\alpha\beta(x)\big],\alpha^2\beta(y),\alpha^3(z)\Big) + as\Big(\big[\beta^2(y)\alpha\beta(z)\big],\alpha^2\beta(w),\alpha^3(x)\Big) \\
   &=& (4f-F)(w,x,y,z) \quad (\text{by Definition \ref{F} and Lemma \ref{propriet ass BiH-A}}) \\
   &=&  f(w,x,y,z) \quad (\text{since $F=3f$}).
\end{eqnarray*}
 This proves that $f$ is equal to the last entry in \eqref{f'}.
\end{proof}


\section{BiHom-alternative algebras are BiHom-Jordan-admissible}
\label{sec:jordan}
In this section, we define BiHom-Jordan(-admissible) algebras. The main result of this section is Theorem \ref{jordan admissible}, which says that BiHom-alternative algebras are BiHom-Jordan-admissible.  Then we give construction results for BiHom-Jordan and BiHom-Jordan-admissible algebras.


\begin{df}\label{BiHom jordan}
A BiHom algebra $(A,\mu,\alpha,\beta)$ is called a BiHom-Jordan algebra if:
\begin{itemize}
  \item [(i)] $\alpha   \beta = \beta   \alpha$,
  \item [(ii)] $\mu \Big(\beta(x),\alpha(y)\Big)=\mu\Big(\beta(y),\alpha(x)\Big),\ \forall x,y \in A$ (BiHom-commutativity condition),
  \item [(iii)] $as_{\alpha,\beta}\Big(\mu\big(\beta^2(x),\alpha\beta(x)\big),\alpha^2\beta(y),\alpha^3(x)\Big)=0,\ \forall x,y \in A$.(BiHom-Jordan identity)
\end{itemize}
\end{df}
Note that if $\alpha=\beta=id$, we obtain a Jordan algebra. Then we conclude that if $(A,\mu)$ is a Jordan algebra, $(A,\mu,id,id)$ can be viewed as a BiHom-Jordan algebra.

The BiHom-Jordan identity can be writen as:  $$\circlearrowleft_{x,w,z}as_{\alpha,\beta}\Big(\mu\big(\beta^2(x),\alpha\beta(w)\big),\alpha^2\beta(y),\alpha^3(z)\Big)=0.$$

\begin{df}
\label{def:plushom}
Let $(A,\mu,\alpha,\beta)$ be a BiHom-algebra.  Define its plus BiHom-algebras as the BiHom-algebra $A^+ = (A,\ast,\alpha,\beta)$, where
$$x\ast y=\frac{1}{2}\Big( \mu(x,y)+\mu\big(\alpha^{-1}\beta(y),\alpha\beta^{-1}(x)\big)\Big).$$
\end{df}
Note that product $\ast$ is BiHom-commutative. Indeed, given $x,y \in A$, then we have
$$\beta(x)\ast\alpha(y)  =\displaystyle \frac{1}{2}\Big( \beta(x)\alpha(y)+\beta(y)\alpha(x)\Big)
   = \displaystyle \frac{1}{2} \Big(\beta(y)\alpha(x)+\beta(x)\alpha(y)\Big)=\beta(y)\ast\alpha(x).$$

A BiHom-Jordan-admissible algebra is a BiHom-algebra $(A,\mu,\alpha,\beta)$ whose plus BiHom-algebra $A^+ = (A,\ast,\alpha,\beta)$ is a BiHom-Jordan algebra.

\begin{thm}
Let $(A,\mu)$ be a Jordan algebra and let $\alpha, \beta: A\rightarrow A$ be two commuting algebra morphisms. Then $(A,\mu'=\mu\circ(\alpha\otimes \beta),\alpha,\beta)$ is a BiHom-Jordan algebra.

\end{thm}

\begin{proof}
First, note that $\mu$ is commutative since $A$ is a Jordan algebra. Let $x,y \in A$, then
$$\mu'(\beta(x),\alpha(y))=\mu(\alpha\beta(x),\alpha\beta(y))=\mu(\alpha\beta(y),\alpha\beta(x))=\mu'(\beta(y),\alpha(x)).$$
On the other hand,
\begin{eqnarray*}
  & &as_{\alpha,\beta}\Big(\mu'\big(\beta^2(x),\alpha\beta(x)\big),\alpha^2\beta(y),\alpha^3(x)\Big) =as_{\alpha,\beta}\Big(\mu\big(\alpha\beta^2(x),\alpha\beta^2(x)\big),\alpha^2\beta(y),\alpha^3(x)\Big) \\
   &=&\mu'\Big( \mu'\big( \mu\big(\alpha\beta^2(x),\alpha\beta^2(x)\big),\alpha^2\beta(y)\big),\alpha^3\beta(x)\Big)
   -\mu'\Big( \mu\big(\alpha^2\beta^2(x),\alpha^2\beta^2(x)\big),\mu'\big(\alpha^2\beta(y),\alpha^3(x)\big)\Big)\\
   &=& \mu'\Big( \mu\big( \mu\big(\alpha^2\beta^2(x),\alpha^2\beta^2(x)\big),\alpha^2\beta^2(y)\big),\alpha^3\beta(x)\Big)
    -\mu'\Big( \mu\big(\alpha^2\beta^2(x),\alpha^2\beta^2(x)\big),\mu\big(\alpha^3\beta(y),\alpha^3\beta(x)\big)\Big)\\
   &=&\mu\Big( \mu\big( \mu\big(\alpha^3\beta^2(x),\alpha^3\beta^2(x)\big),\alpha^3\beta^2(y)\big),\alpha^3\beta^2(x)\Big)
   -\mu\Big( \mu\big(\alpha^3\beta^2(x),\alpha^3\beta^2(x)\big),\mu\big(\alpha^3\beta^2(y),\alpha^3\beta^2(x)\big)\Big)\\
   &=& \alpha^3\beta^2 \Big(as(\mu(x,x),y,x)\Big)=0.
\end{eqnarray*}
Therefore  $(A,\mu'=\mu\circ(\alpha\otimes \beta),\alpha,\beta)$ is a BiHom-Jordan algebra.

\end{proof}
\begin{rem}
Similarly to previous theorem, one may construct a new BiHom-Jordan algebra starting with a given BiHom-Jordan algebra and a pair of commuting BiHom-algebra morphisms.
\end{rem}

Now, we give the main result of this section

\begin{thm}
\label{jordan admissible}
Every BiHom-alternative algebra is BiHom-Jordan-admissible.
\end{thm}

To prove Theorem \ref{jordan admissible}, we will use the following preliminary observations.

\begin{lem}
\label{lem thm jordan admis}
Let $(A,\mu,\alpha,\beta)$ be a regular BiHom-alternative algebra. Then, for each $x,y \in A$ we have  the following statements
\begin{itemize}
  \item [(i)]  $as\big(\beta^2(x)\alpha\beta(x), \alpha^2\beta(y),\alpha^3(x)\big)=0$.
  \item [(ii)] $\Big(\alpha\beta^2(y)\big(\alpha\beta(x)\alpha^2(x)\big)\Big)\alpha^3\beta(x)=
      \Big(\alpha\beta^2(y)\alpha^2\beta(x)\Big)\Big(\alpha^2\beta(x)\alpha^3(x)\Big)$.
  \item [(iii)] $\alpha^2\beta^2(x)\Big(\big(\alpha\beta(x)\alpha^2(x)\big)\alpha^3(y)\Big)=
      \Big(\alpha\beta^2(x)\alpha^2\beta(x)\Big)\Big(\alpha^2\beta(x)\alpha^3(y)\Big)$.
  \item [(iv)]  $\alpha^2\beta^2(x)\Big(\alpha^2\beta(y)\big(\alpha^2(x)\alpha^3\beta^{-1}(x)\big)\Big)=
      \Big(\alpha\beta^2(x)\alpha^2\beta(y)\Big)\Big(\alpha^2\beta(x)\alpha^3(x)\Big)$.
\end{itemize}

\end{lem}

\begin{proof}
(i)- For each $x,y \in A$, we have from  Definition \ref{f}
\begin{eqnarray*}
   & &as\big(\beta^2(x)\alpha\beta(x), \alpha^2\beta(y),\alpha^3(x)\big)  \\
   &=& f(x,x,y,x) +as\big(\beta^2(x),\alpha\beta(y),\alpha^2(x)\big)\alpha^3\beta(x)+\alpha^2\beta^2(x)as\big(\alpha\beta(x),\alpha^2(y),\alpha^3\beta^{-1}(x)\big)\\
   &=& \alpha^2\beta^2(x)as\big(\alpha\beta(x),\alpha^2(y),\alpha^3\beta^{-1}(x)\big)\quad  (\text{by Lemma \ref{propriet ass BiH-A} and Proposition \ref{prop:f alternating}}) \\
   &=& \alpha^2\beta^2(x)as\Big(\beta^2\big(\alpha\beta^{-1}(x)\big),\alpha\beta\big(\alpha\beta^{-1}(y)\big),\alpha^2\big(\alpha\beta^{-1}(x)\big)\Big) \\
   &=&  \alpha^2\beta^2(x)\alpha\beta^{-1}\Big(as\big(\beta^2(x),\alpha\beta(y),\alpha^2(x)\big)\Big)= 0 \quad \text{(by Lemma \ref{propriet ass BiH-A})}.
\end{eqnarray*}
(ii)- Let $x, y \in A$. Since the function $f$ defined in \eqref{f} is alternating (Proposition \ref{prop:f alternating}), then $f(y,x,x,x)=0$.
That is
\begin{eqnarray*}
   & &as\big(\beta^2(y)\alpha\beta(x), \alpha^2\beta(x),\alpha^3(x)\big)  \\ &=&as\big(\beta^2(x),\alpha\beta(x),\alpha^2(x)\big)\alpha^3\beta(y)+\alpha^2\beta^2(x)as\big(\alpha\beta(y),\alpha^2(x),\alpha^3\beta^{-1}(x)\big)\\
   &=&0\quad  (\text{by Lemma \ref{propriet ass BiH-A}}),
\end{eqnarray*}
which implies that
$$\Big(\big(\beta^2(y)\alpha\beta(x)\big)\alpha^2\beta(x)\Big)\alpha^3\beta(x)=\Big(\alpha\beta^2(y)\alpha^2\beta(x)\Big)\Big(\alpha^2\beta(x)\alpha^3(x)\Big).$$
Now, using the fact that $A$ is a BiHom-alternative algebra, therefore,
$$as\big(\beta^2(y),\alpha\beta(x),\alpha^2(x)\big)=0,$$
which implies
$$\big(\beta^2(y)\alpha\beta(x)\big)\alpha^2\beta(x)=\alpha\beta^2(y)\big(\alpha\beta(x)\alpha^2(x)\big).$$
Thus
 $$\Big(\alpha\beta^2(y)\big(\alpha\beta(x)\alpha^2(x)\big)\Big)\alpha^3\beta(x)=\Big(\alpha\beta^2(y)\alpha^2\beta(x)\Big)\Big(\alpha^2\beta(x)\alpha^3(x)\Big).$$
(iii)- Given $x, y \in A$, one has
\begin{eqnarray*}
  as (\alpha\beta^2(x),\alpha\beta(x)\alpha^2(x),\alpha^3(y)\big) &=& as\big(\beta^2(\alpha(x)),\alpha\beta(x\alpha\beta^{-1}(x)),\alpha^2(\alpha(y))\big) \\
   &=& as\big(\beta^2(x)\alpha\beta(x), \alpha^2\beta(y),\alpha^3(x)\big)=0 \quad (\text{using Lemma \ref{propriet ass BiH-A} and (i)}).
\end{eqnarray*}
Then
\begin{eqnarray*}
  \alpha^2\beta^2(x)\Big(\big(\alpha\beta(x)\alpha^2(x)\big)\alpha^3(y)\Big) &=& \Big(\alpha\beta^2(x)\big(\alpha\beta(x)\alpha^2(x)\big)\Big)\alpha^3\beta(y) \\
   &=& \Big(\big(\beta^2(x)\alpha\beta(x)\big)\alpha^2\beta(x)\Big)\alpha^3\beta(y)  \quad (\text{since $A$ is BiHom-alternative}) \\
   &=& \Big(\alpha\beta^2(x)\alpha^2\beta(x)\Big)\Big(\alpha^2\beta(x)\alpha^3(y)\Big),
\end{eqnarray*}
since $as\big(\beta^2(x)\alpha\beta(x), \alpha^2\beta(x),\alpha^3(y)\big)=0$.
This prove the identity (iii).

(iv)- Taken two elements $x, y \in A$. Then we have
\begin{eqnarray*}
  as\big(\alpha\beta^2(x),\alpha^2\beta(y),\alpha^2(x)\alpha^3\beta^{-1}(x)\big) &=& as\big(\beta^2(\alpha(x)),\alpha\beta(\alpha(y)),\alpha^2(x\alpha\beta^{-1}(x))\big) \\
   &=& as\big(\beta^2(x)\alpha\beta(x), \alpha^2\beta(x),\alpha^3(y)\big)=0.
\end{eqnarray*}
Thus
 $$\alpha^2\beta^2(x)\Big(\alpha^2\beta(y)\big(\alpha^2(x)\alpha^3\beta^{-1}(x)\big)\Big)=
 \Big(\alpha\beta^2(x)\alpha^2\beta(y)\Big)\Big(\alpha^2\beta(x)\alpha^3(x)\Big).$$
\end{proof}
Now, we  prove Theorem \ref{jordan admissible}.
\begin{proof}
Let $(A,\mu,\alpha,\beta)$ be a regular BiHom-alternative algebra and let $(A^+,\ast,\alpha,\beta)$ be its plus BiHom-algebra. We have to show the BiHom-Jordan identity in $A^+$. First note that
$$\beta^2(x)\ast\alpha\beta(x)=\frac{1}{2}\Big(\beta^2(x)\alpha\beta(x)+\beta^2(x)\alpha\beta(x)\Big)=\beta^2(x)\alpha\beta(x).$$
Then
\begin{eqnarray*}
& & 4 as_{A^+}\big(\beta^2(x)\ast\alpha\beta(x), \alpha^2\beta(y),\alpha^3(x)\big) =4 as_{A^+}\big(\beta^2(x)\alpha\beta(x), \alpha^2\beta(y),\alpha^3(x)\big) \\
   &=&\Big(\big(\beta^2(x)\alpha\beta(x)\big)\ast\alpha^2\beta(y)\Big)\ast\alpha^3\beta(x)-
   \Big(\alpha\beta^2(x)\alpha^2\beta(x)\Big)\ast\Big(\alpha^2\beta(y)\ast\alpha^3(x)\Big)  \\
   &=& \Big(\big(\beta^2(x)\alpha\beta(x)\big)\alpha^2\beta(y) +\alpha\beta^2(y)\big(\alpha\beta(x)\alpha^2(x)\big)\Big)\ast\alpha^3\beta(x)\\
   & -&  \Big(\alpha\beta^2(x)\alpha^2\beta(x)\Big)\ast\Big(\alpha^2\beta(y)\alpha^3(x)+\alpha^2\beta(x)\alpha^3(y)\Big)   \\
   &=& \Big(\big(\beta^2(x)\alpha\beta(x)\big)\alpha^2\beta(y)\Big)\alpha^3\beta(x)+ \Big(\alpha\beta^2(y)\big(\alpha\beta(x)\alpha^2(x)\big)\Big)\alpha^3\beta(x) \\
   &+& \alpha^2\beta^2(x)\Big(\big(\alpha\beta(x)\alpha^2(x)\big)\alpha^3(y)\Big)+ \alpha^2\beta^2(x)\Big(\alpha^2\beta(y)\big(\alpha^2(x)\alpha^3\beta^{-1}(x)\big)\Big) \\
   &-& \Big(\alpha\beta^2(x)\alpha^2\beta(x)\Big)\Big(\alpha^2\beta(y)\alpha^3(x)\Big)-
   \Big(\alpha\beta^2(x)\alpha^2\beta(x)\Big)\Big(\alpha^2\beta(x)\alpha^3(y)\Big) \\
   &-& \Big(\alpha\beta^2(y)\alpha^2\beta(x)\Big)\Big(\alpha^2\beta(x)\alpha^3(x)\Big)-
   \Big(\alpha\beta^2(x)\alpha^2\beta(y)\Big)\Big(\alpha^2\beta(x)\alpha^3(x)\Big)  \\
   &=& as\big(\beta^2(x)\alpha\beta(x), \alpha^2\beta(y),\alpha^3(x)\big)  \\
   &+& \Big(\alpha\beta^2(y)\big(\alpha\beta(x)\alpha^2(x)\big)\Big)\alpha^3\beta(x)-
      \Big(\alpha\beta^2(y)\alpha^2\beta(x)\Big)\Big(\alpha^2\beta(x)\alpha^3(x)\Big)\\
   &+& \alpha^2\beta^2(x)\Big(\big(\alpha\beta(x)\alpha^2(x)\big)\alpha^3(y)\Big)-
      \Big(\alpha\beta^2(x)\alpha^2\beta(x)\Big)\Big(\alpha^2\beta(x)\alpha^3(y)\Big)\\
   &+& \alpha^2\beta^2(x)\Big(\alpha^2\beta(y)\big(\alpha^2(x)\alpha^3\beta^{-1}(x)\big)\Big)-
      \Big(\alpha\beta^2(x)\alpha^2\beta(y)\Big)\Big(\alpha^2\beta(x)\alpha^3(x)\Big) \\
   &=& 0 \qquad (\text{by Lemma \ref{lem thm jordan admis}} ).
\end{eqnarray*}
Therefore, $A^+$ is a BiHom-Jordan algebra which implies that $A$ is BiHom-Jordan-admissible.
\end{proof}

\begin{prop}
Let $(A,\mu)$ be a Jordan-admissible algebra and $\alpha, \beta:A\rightarrow A$ be two algebra morphisms such that $\alpha  \beta  =\beta  \alpha$. Then the induced BiHom-algebra $A_{\alpha,\beta}=(A,\mu'=\mu( \alpha \otimes\beta),\alpha,\beta)$ is a BiHom-Jordan-admissible algebra.

\end{prop}

\begin{proof}
We must prove that $(A,\ast,\a,\b)$ is  a BiHom-Jordan algebra where $$x\ast y=\frac{1}{2}\Big( \mu'(x,y)+\mu'\big(\alpha^{-1}\beta(y),\alpha\beta^{-1}(x)\big)\Big)$$
First, note that $(A,\circ)$ is a Jordan algebra, where $\circ=\dis\frac{1}{2}(\mu+\mu^{op})$.\\
In addition, it is easy to check that $x \ast y=\a(x)\circ \b(y)$. Then,
$$\b(x)\ast \a(y)=\a\b(x)\circ \b\a(y)=\b\a(y) \circ \a\b(x)=\b(y) \ast \a(x).$$
On the other hand,
\begin{eqnarray*}
   & & as_{\a,\b}(\b^2(x)\ast \a\b(x),\a^2\b(y),\a^3(x)) \\
  &=& ((\a\b^2(x)\circ \a^2\b(x))\ast \a^2\b(y))\ast \a^3\b(x)-(\a^2\b^2(x)\circ \a^2\b^2(x))\ast (\a^2\b(y)\ast\a^3(x))  \\
   &=& \a^3\b^2\big( \underbrace{((x\circ x)\circ y)\circ x-(x \circ x)\circ (y \circ x)}_{=0}  \big)\ \text{(since}\ (A,\circ)\  \text{is a Jordan algebra})  \\
   &=& 0.
\end{eqnarray*}
This finishes the proof.
\end{proof}





\begin{thebibliography}{AA}

\bibitem{KadriAbdelkaderAbdenacer}
 E. Abdaoui,, A. Ben Hassine, and A. Makhlouf. \emph{BiHom-Lie colour algebras structures}. arXiv preprint arXiv:1706.02188 (2017).

\bibitem{albert}
A.A. Albert, \emph{A structure theory for Jordan algebras}, Ann. Math. 48 (1947) 546--567.


\bibitem{albert2}
A.A. Albert, \emph{Power-associative rings}, Trans. Amer. Math. Soc. 64 (1948) 552--593.

\bibitem{albert3}
A.A. Albert, \emph{On right alternative algebras}, Ann. Math. 50 (1949) 318-328.
\bibitem{mikheev}
I.M. Mikheev, \emph{On an identity in right alternative rings}. Alg. i Logika 8 (1969) 357-366; english translation:  Alg. Logic 8 (1969) 204--211.



\bibitem{ama}
F. Ammar and A. Makhlouf, \emph{Hom-Lie superalgebras and Hom-Lie admissible superalgebras}.  J. Algebra \textbf{324} (2010), no. 7, 1513--1528.

\bibitem{baez}
J.C. Baez, \emph{The octonions}. Bull. Amer. Math. Soc. 39 (2002) 145--205.

\bibitem{bk}
R.H. Bruck and E. Kleinfeld, \emph{The structure of alternative division rings}, Proc. Amer. Math. Soc. 2 (1951) 878--890.





\bibitem{GrazianiMakhloufMeniniPanaite}
G. Graziani, A. Makhlouf, C. Menini, F.  Panaite, \emph{BiHom-associative algebras, BiHom-Lie algebras and BiHom-bialgebras}. SIGMA Symmetry Integrability Geom. Methods Appl. 11 (2015), Paper 086, 34 pp.
\bibitem{gt}
F. G\"{u}rsey and C.-H. Tze, On the role of division, \emph{Jordan and related algebras in particle physics}, World Scientific, Singapore, 1996.
\bibitem{GuoZhangWang}
Guo, S., Zhang, X.,  Wang, S. (2017). \emph{The construction and deformation of BiHom-Novikov algebras}. J. Geom. Phys. 132 (2018), 460--472.

\bibitem{hls}
J.T. Hartwig, D. Larsson, and S.D. Silvestrov, \emph{Deformations of Lie algebras using $\sigma$-derivations}, J. Algebra 295 (2006) 314--361.

\bibitem{jacobson}
N. Jacobson, \emph{Structure and representations of Jordan algebras}, Amer. Math. Soc., Providence, RI, 1968.

\bibitem{jvw}
P. Jordan, J. von Neumann, and E. Wigner, \emph{On an algebraic generalization of the quantum mechanical formalism}, Ann. Math. 35 (1934) 29--64.

\bibitem{kerdman}
F.S. Kerdman, \emph{Analytic Moufang loops in the large}.  Alg. Logic 18 (1980) 325--347.

\bibitem{kuzmin}
E.N. Kuz'min, \emph{The connection between Mal'cev algebras and analytic Moufang loops}.  Alg. Logic 10 (1971) 1--14.
\bibitem{LarssonSilvestrov}
D. Larsson and S. D. Silvestrov, \emph{Quasi-Hom-Lie algebras, Central Extensions and 2-cocycle-like identities}. J. of
Algebra, 288 (2005), 321--344.




\bibitem{mak}
A. Makhlouf, \emph{Hom-alternative algebras and Hom-Jordan algebras}. Int. Electron. J. Algebra 8 (2010), 177--190.

\bibitem{mak2}
A. Makhlouf, \emph{Paradigm of nonassociative Hom-algebras and Hom-superalgebras}. Proceedings of Jordan Structures in Algebra and Analysis Meeting, 143--177, Editorial Circulo Rojo, Almeria, 2010.

\bibitem{ms}
A. Makhlouf and S. Silvestrov, \emph{Hom-algebra structures}, J. Gen. Lie Theory Appl. 2 (2008) 51--64.

\bibitem{ms4}
A. Makhlouf and S. Silvestrov, \emph{Hom-algebras and Hom-coalgebras}. J. Algebra Appl. 9 (2010), no. 4, 553--589.

\bibitem{maltsev}
A.I. Mal'tsev, \emph{Analytic loops}. Mat. Sb. 36 (1955), 569--576.

\bibitem{moufang}
R. Moufang, \emph{Zur struktur von alternativk\"{o}rpern}. Math. Ann. 110 (1935), 416--430.

\bibitem{myung}
H.C. Myung, \emph{Malcev-admissible algebras}. Progress in Math. 64, Birkh\"{a}user, Boston, MA, 1986.

\bibitem{nagy}
P.T. Nagy, \emph{Moufang loops and Malcev algebras}. Sem. Sophus Lie 3 (1993), 65--68.

\bibitem{okubo}
S. Okubo, \emph{Introduction to octonion and other non-associative algebras in physics}, Cambridge Univ. Press, Cambridge, UK, 1995.

\bibitem{ps}
J.M. P\'{e}rez-Izquierdo and I.P. Shestakov, \emph{An envelope for Malcev algebras}. J. Algebra 272 (2004) 379--393.

\bibitem{sabinin}
L.V. Sabinin, Smooth quasigroups and loops. Kluwer Academic, The Netherlands, 1999.

\bibitem{sagle}
A.A. Sagle, \emph{Malcev algebras}. Trans. Amer. Math. Soc. 101 (1961) 426--458.

\bibitem{schafer}
R.D. Schafer, \emph{An introduction to nonassociative algebras}. Dover, New York, 1996.

\bibitem{sv}
T.A. Springer and F.D. Veldkamp, \emph{Octonions. Jordan algebras, and exceptional groups}. Springer Verlag, Berlin, 2000.

\bibitem{tw}
J. Tits and R.M. Weiss, \emph{Moufang polygons}. Springer-Verlag, Berlin, 2002.


\bibitem{yau2}
D. Yau, \emph{Hom-algebras and homology}, J. Lie Theory 19 (2009), 409--421.


\bibitem{Yau3}
 D. Yau, \emph{Hom-Maltsev, Hom-alternative and Hom-Jordan algebras}. International Electronic Journal of algebras,
11 (2012), 177--217.


%
%
%
%
%
%
%

\end{thebibliography}
\end{document}